\documentclass[a4paper,11pt]{amsart}
\usepackage{amsmath}
\usepackage{graphicx}
\usepackage{epstopdf}
\usepackage{amssymb,amscd,array}
\usepackage{color}
\usepackage{float}

\makeindex \makeatletter
\def\captionof#1#2{{\def\@captype{#1}#2}}
\makeatother

\newcounter{tablegroup}
\newcounter{subtable}[tablegroup]

\newtheorem{thm}{Theorem}[section]
\newtheorem{cor}[thm]{Corollary}
\newtheorem{lem}[thm]{Lemma}
\newtheorem{prop}[thm]{Proposition}
\newtheorem{defn}[thm]{Definition}
\newtheorem{rem}[thm]{\bf Remark}
\newtheorem{exe}[thm]{\bf Example}
\numberwithin{equation}{section}

\begin{document}
\title[Nonwandering sets and Special $\alpha$-limit sets]
{Nonwandering sets and Special $\alpha$-limit sets \\ of Monotone maps on Regular Curves}

\author{Aymen Daghar and Habib Marzougui}
\address{Aymen Daghar, \ University of Carthage, Faculty
	of Science of Bizerte, (UR17ES21), ``Dynamical systems and their applications'', Jarzouna, 7021, Bizerte, Tunisia.}
\email{aymendaghar@gmail.com}
\address{ Habib Marzougui, University of Carthage, Faculty
	of Science of Bizerte, (UR17ES21), ``Dynamical systems and their applications'', 7021, Jarzouna, Bizerte, Tunisia}
\email{habib.marzougui@fsb.rnu.tn}

\subjclass[2010]{ 37B20, 37B45, 54H20}
\keywords{ Minimal set, regular curve, $\omega$-limit set, local dendrite, monotone map}
	\begin{abstract}
	
		Let $X$ be a regular curve and let $f: X\to X$ be a monotone map. In this paper, nonwandering set of $f$ and the structure of special $\alpha$-limit sets
		for $f$ are investigated. We show that AP$(f)= \textrm{R}(f) =\Omega(f)$, where AP$(f)$, $\textrm{R}(f)$ and $\Omega(f)$ are the sets of almost periodic points, recurrent points and nonwandering of $f$, respectively. This result extends that of Naghmouchi established, whenever $f$ is a homeomorphism on a regular curve [J. Difference Equ. Appl., 23 (2017), 1485--1490] and [Colloquium Math., 162 (2020), 263--277], and that of Abdelli and Abdelli, Abouda and Marzougui, whenever $f$ is a monotone map on a local dendrite [Chaos, Solitons Fractals, 71 (2015), 66--72] and [Topology Appl., 250 (2018), 61--73], respectively. On the other hand, we show that for every $X\setminus \textrm{P}(f)$, the special $\alpha$-limit set $s\alpha_{f}(x)$ is a minimal set, where P$(f)$ is the set of periodic points of $f$ and that $s\alpha_{f}(x)$ is always closed, for every $x\in X$. In addition, we prove that $\textrm{SA}(f) = \textrm{R}(f)$, where  $\textrm{SA}(f)$ denotes the union of all special $\alpha$-limit sets of $f$; these results extend, for monotone case, recent results on interval and graph maps obtained respectively by Hant\'{a}kov\'{a} and Roth in [Preprint: arXiv 2007.10883.] and Fory\'{s}-Krawiec,  Hant\'{a}kov\'{a} and Oprocha in [Preprint: arXiv:2106.05539.]. Further results related to the continuity of the limit maps are also obtained, we prove that the map $\omega_{f}$ (resp. $\alpha_{f}$, resp. s$\alpha_{f}$) is continuous on $X\setminus \textrm{P}(f)$ (resp. $X_{\infty}\setminus \textrm{P}(f)$). 
	\end{abstract}
\maketitle

\section{\bf Introduction}
In the last two decades, a wide literature on the dynamical properties of maps on some one-dimensional continua has developed.  Examples of continua become increasingly studied include, graphs, dendrites and local dendrites (see for instance \cite{a}, \cite{HM}, \cite{aam}, \cite{Nag3}).

Recently a large class of continua called \textit{regular curves} has given a special attention (see for example, \cite{an}, \cite{Ay}, \cite{Ay3}, \cite{MONOTONE}, \cite{ka}, \cite{ka2}, \cite{n}, \cite{n2}, \cite{Seidler}).
These form a large class of continua which includes local dendrites.
The Sierpi\'{n}ski triangle is a well known example of a regular curve which is not a local dendrite.  Regular curves appear in continuum theory and also in other branches of Mathematics such as complex dynamics; for instance, the Sierpi\'{n}ski triangle can be realized as the Julia set of the complex polynomial $p(z)=z^2+2$ (see \cite{Bldevan}).
 Recall that the Julia set of $p$ is $J(p) = \{z\in \mathbb{C}: (p^n(z))_{n\geq 1} \textrm{ is bounded} \}$.
Seidler \cite{Seidler} proved that every homeomorphism of a regular curve has zero topological entropy (later, this result was extended by Kato in \cite{ka} to monotone maps).
In \cite{n}, \cite{n2}, Naghmouchi proved that any $\omega$-limit set (resp. $\alpha$-limit set) of a homeomorphism $f$ on a regular curve is a minimal set. Moreover he established the equality between the set of nonwandering points and the set of almost periodic points.
~ In \cite{Ay}, the first author gave a full characterization of minimal sets for homeomorphisms without periodic points on regular curves. In \cite{MONOTONE2} it was shown that the set of periodic
points is either empty or dense in the set of non-wandering points, for homeomorphisms on regular curves.
In the present paper, we deal with several questions/problems.

First, we address the question of the equality between the set of nonwandering points and the set of almost periodic points.
This was proved in two cases:

$-$ for homeomorphisms on regular curves (Naghmouchi \cite{n2})

$-$ for monotone maps on local dendrites (Abdelli et al. \cite{aam})

In Theorem \ref{t41}, we prove a more general result by showing that this is true for monotone maps on regular curves.
Nevertheless, the later result is false in general for monotone maps on dendroids as it is shown in Example \ref{CEM2}. Moreover, we show in Theorem \ref{RR} that the set of nonwandering points coincides with the set of points belonging to their special $\alpha$-limit sets (Proposition \ref{RRS}).\\

Second, beside the usual limits sets $\omega$-limit and $\alpha$-limit, we are interested in the study of another kind of limit sets, called \textit{special $\alpha$-limit set} (see Section 6). We ask the question whether every special $\alpha$-limit set is a minimal set? We show that for a monotone map $f$ on a regular curve, every special $\alpha$-limit set $s\alpha_{f}(x)$ of a non periodic point $x$ is a minimal set. However, we built an example of a monotone map $f$ on an infinite star for which $s\alpha_{f}(x)$ is not a minimal set for some periodic point $x$. Moreover we show that the inclusion $s\alpha_{f}(x)\subset \alpha_{f}(x)$ is strict. In addition, we prove that $\textrm{SA}(f) = \textrm{R}(f)$, where $\textrm{SA}(f)$ respectively $\textrm{R}(f)$ denotes the union of all special $\alpha$-limit sets and the set of recurrent points of $f$. Notice that it is shown recently that, for mixing graph maps $f: G\to G$, every special $s\alpha$-limit set is the $\omega$-limit set of some point from $G$ and moreover, for graph maps $f: G\to G$ with zero
topological entropy, every special $s\alpha$-limit set is a minimal set (cf. \cite[Theorem 3.8]{kho}). On the other hand, Kolyada et al. showed in (\cite[Theorem 3.3 and Corollary 3.11]{kms}),  that $s\alpha_{f}(x)$ is closed whenever $f$ is either an interval map for which the set of all periodic points is closed or $f$ is transitive. Recently, Hant\'{a}kov\'{a} and Roth \cite{hr} showed that a special $\alpha$-limit set is closed for a piecewise monotone interval map. It is previously known that  $s\alpha_{f}(x)$ is closed for homeomorphisms on any compact metric space.
In Corollary \ref{c53}, we extend the later result to monotone maps on regular curves by showing that for any $x\in X$, $s\alpha_{f}(x)$ is closed. In fact we show more precisely that for every $x\in X_{\infty}$, $\alpha_{f}(x)\cap \Omega(f) = s\alpha_{f}(x)$ (cf. Theorem \ref{S=OSS}).

\section{\bf Preliminaries}

Let $X$ be a compact metric space with metric $d$ and let $f: X\to X$ be a continuous map. The pair $(X,f)$ is called a dynamical system. Let $\mathbb{Z},\ \mathbb{Z}_{+}$ and $\mathbb{N}$ be the sets of integers, non-negative integers and positive integers, respectively. For $n\in \mathbb{Z_{+}}$, denote by  $f^{n}$ the $n$-$\textrm{th}$ iterate of $f$; that is, $f^{0}=\textrm{identity}$ and $f^{n}=f\circ f^{n-1}$ if $n\in \mathbb{N}$. For any $x\in X$, the subset
$\textrm{Orb}_{f}(x) = \{f^{n}(x): n\in\mathbb{Z}_{+}\}$ is called the \textit{orbit} of $x$ (under $f$). A subset $A\subset X$ is called \textit{$f-$invariant} (resp. strongly $f-$invariant) if $f(A)\subset A$ (resp., $f(A)=A$); it is further called a \textit{minimal set} (under $f$) if it is closed, non-empty and does not contain any $f-$invariant, closed proper non-empty subset of $X$. We define the \textit{$\omega$-limit} set of a point $x\in X$ to be the set:
\begin{align*}
\omega_{f}(x) & = \{y\in X: \liminf_{n\to +\infty} d(f^{n}(x), y) = 0\}\\
& = \underset{n\in \mathbb{N}}\cap\overline{\{f^{k}(x): k\geq n\}}.
\end{align*}

 The set $\alpha_{f}(x)= \displaystyle\bigcap_{k\geq 0}\displaystyle\overline{\bigcup_{n\geq k}f^{-n}(x)}$ is called the $\alpha$-limit set of $x$. Equivalently a point $y\in \alpha_{f}(x)$ if and only if there exist an increasing sequence of positive integers $(n_{k})_{k\in \mathbb{N}}$ and a sequence of points  $(x_{k})_{k\geq 0}$ such that
$f^{n_{k}}(x_{k})=x$ and $\displaystyle\lim_{k\rightarrow +\infty}x_{k}=y.$  It is much harder to deal with $\alpha$-limit sets since there are many choices for points in a backward orbit. When $f$ is a homeomorphism, $\alpha_{f}(x)= \omega_{f^{-1}}(x)$.
\medskip

Balibrea et al. \cite{bgl} considered exactly one branch of the backward orbit as follows.

\begin{defn}
	Let $x\in X$. A \textit{negative orbit} of $x$ is a sequence $(x_n)_n$ of points in $X$ such that $x_0=x$ and $f(x_{n+1}) = x_n$, for every $n\geq 0$. The $\alpha$-limit set of $(x_n)_{n\geq 0}$ denoted by $\alpha_f((x_n)_n)$ is the set of all limit points of $(x_n)_{n\geq 0}$.
\end{defn}

It is clear that $\alpha_f((x_n)_n)\subset \alpha_{f}(x)$. The inclusion can be strict; even for onto monotone maps on regular curves (see Example \ref{e616}). Notice that we have the following equivalence:\\

(i) $\alpha_{f}(x)\neq \emptyset$, (ii) $x\in X_{\infty}$.
In particular, if $f$ is onto, then $\alpha_{f}(x)\neq \emptyset$ for every $x\in X$.

\begin{prop}$($\cite[Corollary 2.2]{MONOTONE}$)$\label{p43}
	Let $f: X\longrightarrow X$ be a continuous self mapping of a compact space $X$. Let $x\in X_{\infty}$. Then:
	\begin{itemize} 	
		\item[(i)] $\alpha_{f}(x)$ is non-empty, closed and $f$-invariant. In addition, it is strongly $f$-invariant whenever $\displaystyle\lim_{n\to +\infty}\mathrm{diam}(f^{-n}(x)) =0$.
		\item[(ii)] $\alpha_f((x_n)_{n\geq 0})$ is non-empty, closed and strongly $f$-invariant, for any negative orbit $(x_{n})_{n\geq 0}$ of $x$.
	\end{itemize}
\end{prop}
\medskip

A point $x\in X$ is called:

$-$ \textit{Periodic} of period $n\in\mathbb{N}$ if
$f^{n}(x)=x$ and $f^{i}(x)\neq x$ for $1\leq i\leq n-1$; if $n = 1$, $x$ is called a \textit{fixed point} of $f$ i.e.
$f(x) = x$;


$-$ \textit{Almost periodic} if for any neighborhood $U$ of $x$ there is $N\in\mathbb{N}$ such that
$\{f^{i+k}(x): 0\leq i\leq N\}\cap U\neq \emptyset$, for all $k\in \mathbb{N}$.

$-$\textit{Recurrent} if $x\in \omega_{f}(x)$.


$-$ \textit{Nonwandering } if for any neighborhood $U$ of $x$ there is $n\in \mathbb{N}$ such that $f^n(U)\cap U\neq\emptyset$.

We denote by P$(f)$, AP$(f)$, R$(f)$, $\Omega(f)$ and $\Lambda(f)$ the sets of periodic points, almost periodic points, recurrent points, nonwandering points and the union of all $\omega$-limit sets of $f$, respectively. Define the space

$X_{\infty} = \displaystyle\bigcap_{n\in \mathbb{N}}f^{n}(X)$. From the definition, we have the following inclusions:
$$\textrm{P}(f)\subseteq \textrm{AP}(f)\subseteq \textrm{R}(f)\subseteq \Lambda(f)\subseteq \Omega(f)\subseteq X_{\infty}.$$

\medskip

In the definitions below, we use the terminology from Nadler \cite{Nadler} and Kuratowski \cite{Kur}.

\begin{defn}$($\cite[ page 131]{Kur}$)$\label{def} \textit{Let $X,~ Y$ be two topological spaces.
		A continuous map $f: X\longrightarrow Y$ is said to be \textit{monotone} if for any connected subset $C$ of $Y$, $f^{-1}(C)$ is connected.}
\end{defn}

When $f$ is closed and onto, Definition \ref{def} is equivalent to that the preimage of any point by $f$ is connected (cf. \cite[ page 131]{Kur}). In particular, this holds if $X$ is compact and $f$ is onto. Notice that $f^{n}$ is monotone for every $n\in \mathbb{N}$ when $f$ itself is monotone.
\medskip

A \emph{continuum} is a compact connected metric space.
An \emph{arc} $I$ (resp. a \emph{circle}) is any space homeomorphic to the compact interval $[0, 1]$
(resp. to the unit circle $\mathbb{S}^{1} =\{z\in \mathbb{C}: \ \vert z\vert = 1\}$).
A space is called \textit{degenerate} if it is a single point, otherwise; it is \textit{non-degenerate}. By a \textit{graph} $G$, we mean a continuum which can be written as the union of finitely many arcs such that any two of them are either disjoint or intersect only in one or both of their endpoints.

A \textit{dendrite} is a locally connected continuum which contains no circle.
Every subcontinuum of a dendrite is a dendrite (\cite[Theorem 10.10]{Nadler}) and every connected subset of $D$ is arcwise connected. A \textit{local dendrite} is a continuum every point of which has a dendrite neighborhood. By (\cite[Theorem 4, page 303]{Kur}), a local dendrite is a locally connected continuum containing only a finite number of circles. As a consequence every subcontinuum of a local
dendrite is a local dendrite (\cite[Theorems 1 and 4, page 303]{Kur}). Every graph and every dendrite is a local dendrite.

A \textit{regular curve} is a continuum $X$ with the property that for each point $x\in X$ and each open neighborhood $V$ of $x$ in $X$, there exists an open neighborhood $U$ of $x$ in $V$
such that the boundary set $\partial U$ of $U$ is finite.
Each regular curve is a $1$-dimensional locally connected continuum.
It follows that each regular curve is locally arcwise connected (see \cite{Kur} and \cite{Nadler}, for more details). In particular every local dendrite is a regular curve (cf.
\cite[page 303]{Kur}).

A continuum $X$ is said to be \textit{finitely suslinean continuum} provided that each infinite family of pairwise disjoint continua is null (i.e. for any disjoint open sets $U, V$, only a finite number of elements of the family meet both $U$ and $V$). Note that each regular curve is finitely suslinean (see  \cite{LFS}). A continuum is called \textit{hereditarily locally connected continuum}, written hlc, provided that every subcontinuum of $X$ is locally connected.
 In particular, finitely suslinean continua and (hence regular curves) are hlc (see \cite{Kur}). A continuum $X$ is said to be \textit{rational curve} provided that each point $x$ of $X$ and each open neighborhood $V$ of $x$ in $X$, there exists an open neighborhood $U$ of $x$ in $V$
 such that the boundary set $\partial U$ of $U$ is at most countable. Clearly, regular curve are rational.
\medskip

Let $X$ be a compact metric space. We denote by $2^X$ (resp. $C(X)$) the set of all non-empty compact subsets (resp. compact connected subsets) of $X$. The Hausdorff metric $d_H$ on $2^{X}$ (respectively $C(X)$) is defined as follows: $d_H (A,B) = \max \Big(\sup_{a\in A} d(a,B),~\sup_{b\in B} d(b,A)\Big)$, where $A, B \in 2^X$ (resp. $C(X)$). For $x\in X$ and $M\in 2^X$, $d(x,M) = \inf_{y \in M} d(x,y)$. With this distance, $(C(X), d_H)$ and  $(2^X, d_H)$ are compact metric spaces. Moreover if $X$ is a continuum, then so are $2^{X}$ and $C(X)$ (see \cite{Nadler}, for more details). Let $f: X\to X$ be a continuous map of $X$. We denote by $2^{f}:2^{X}\to 2^{X}, A\to f(A)$, called the induced map. Then $2^{f}$ is also a continuous self mapping of $(2^{X},d_{H})$ (cf. \cite{Nadler}). For a subset $A$ of $X$, we denote by diam$(A) = \sup_{x,y\in A} d(x,y)$ and $\textrm{card}(A)$ the cardinality of $A$. For any $A\in C(X)$, we denote by
$$\textrm{Mesh}(A) = \displaystyle\sup \{\textrm{diam}(C): \ C \textrm{ is a connected component of } A\}.$$ 

A family $(A_{i})_{i\in I}$ of subsets of $X$ is called a \textit{null family} if for any infinite sequence $(i_{n})_{n\geq 0}$ of $I,\; \displaystyle\lim_{n\to +\infty}\textrm{diam}(A_{i_{n}})=0$.
\medskip
It is well known that 
each pairwise disjoint family of subcontinua of a regular curve is null (see \cite{LFS}).
\medskip

We recall some results which are needed for the sequel.

\begin{prop} \label{Fs}
	Let $X$ be a regular curve. Then for any $\varepsilon>0$ and for any family of pairwise disjoint subcontinua $(A_{i})_{i\in I}$ of $X$, the set \\
	$\{i\in I: \mathrm{diam}(A_{i})\geq \varepsilon\}$ is finite. In particular if $(A_{n})_{n\geq 0}$ is a sequence of pairwise disjoint continua, then $(A_{n})_{n\geq 0}$ is a null family.
\end{prop}
%
%
%

%

\begin{defn}$($Weak incompressibility$)$ A set $A \subset X$ is said to have \textit{the
weak incompressibility property} if for any proper closed subset $F\subsetneqq A$ (i.e. $F$ is nonempty and distinct from $A$), we have that $F\cap \overline{f(A\setminus F)}\neq\emptyset$.
\end{defn}
Notice that the term weak incompressibility seems to have
appeared first in \cite{bl}.

\begin{lem}$($\cite[Lemma 3, page 71]{BLOCK}$)$ \label{wc} For any $x\in X$, $\omega_f(x)$ has the weak incompressibility property.
\end{lem}

\begin{lem}\label{l27}  Let $A\subset \mathrm{P}(f)$ be a closed invariant subset of $X$ with the weak incompressibility property. If some $a\in A$ is an isolated point of $A$, then $A = O_{f}(a)$.
\end{lem}

\begin{proof}
Assume that for some $a\in A, \{a\}$ is an open subset of $A$ and $O_{f}(a) \subsetneq A$. Since $f_{\mid A}$ is an homeomorphism, $O_{f}(a)$ is a finite open subset of $A$, thus also a proper closed subset of $A$. $A$ has the weak incompressibility property and $O_{f}(a)$ is a closed subset of $A$, we get $O_{f}(a)\cap \overline{f(A\setminus O_{f}(a) )}\neq\emptyset$. Therefore $O_{f}(a)\cap (A\setminus O_{f}(a))\neq \emptyset$, which contradict the fact that $f$ is one to one on $A$.
\end{proof}
	
Let $f$ be a monotone map on a regular curve $X$. If $M$ is an infinite minimal set of $f$, we call  $\mathcal{B}(M)=\{x\in X:  \omega_f(x)= M\}$ the \textit{basin of attraction of} $M$.
\medskip

\begin{thm}\cite{MONOTONE}\label{PAMM1}
Let $f : X\longrightarrow X$ be a regular curve monotone map. Then the following assertions hold:
\begin{itemize}
	\item[(1)] $\omega_f(x)$ is a minimal set, for all $x\in X$.
\item[(2)] $\alpha_{f}((x_{n})_{n\geq 0})$ is a minimal set any negative orbit $(x_{n})_{n\geq 0}$ of $x\in X_{\infty}$. Moreover if $x\in X_{\infty}\setminus \mathrm{P}(f)$, then $\alpha_{f}(x)$ is a minimal set and $\alpha_{f}(x)= \alpha_{f}((x_{n})_{n\geq 0})$, for any negative orbit $(x_{n})_{n\geq 0}$ of $x$.
\item[(3)] For every infinite minimal set $M$,  $\mathcal{B}(M)$ is a closed subset of $X$.
\item[(4)] For every $x\in X_{\infty}$, if $\omega_{f}(x)$ is infinite, then $\alpha_{f}(x)=\omega_{f}(x)$.
\item[(5)] $\mathrm{AP}(f)=\Lambda(f) = \mathrm{R}(f)$.
\end{itemize}
\end{thm}
\bigskip
\section{\bf Nonwandering sets of monotone maps on regular curves}
The aim of this section is to prove the following theorem which extends  \cite[Theorem 1.1]{aam} and \cite[Theorem 2.1]{n2}.

\begin{thm}\label{t41} Let $X$ be a regular curve and $f$ a monotone self mapping of $X$. Then $\Omega(f)= \mathrm{R}(f)=\Lambda(f)=\mathrm{AP}(f)$.
\end{thm}

\begin{proof} Following Theorem \ref{PAMM1}, (5), it suffices to prove that $\Omega(f)\subset \textrm{R}(f)$. Assume that there exists $x\in \Omega(f)\setminus \textrm{R}(f)$. We distinguish two cases:
	\medskip
	
	\textit{Case 1}: $x\notin \overline{P(f)}$. In this case, we follow similarly the proof given in (\cite[Theorem 2.1]{n2}). Let $V$ be an open neighborhood of $x$ such that $\overline{V}\cap \textrm{P}(f)=\emptyset$. As $x\in \Omega(f)$, 
	there is a sequence $(x_{k})_{k\geq 0}$ in $X$ converging to $x$ and a sequence of positive integers $(n_{k})_{k\geq 0}$ such that $f^{n_{k}}(x_{k})$ converges to $x$. Since $x\notin \textrm{R}(f)$, $x\notin \omega_{f}(x)$. Then we can find an open neighborhood $U_1\subset V$ of $x$ with finite boundary and an open neighborhood $U_2$ of $\omega_{f}(x)$ such that  $U_1\cap U_2=\emptyset$.
	 Since $X$ is locally connected, so by (\cite[Theorem 4, page 257]{Kur}), for $k$ large enough, we can find a sequence of arcs $(I_{k})_{k\geq 0}\subset U_1$ joining $x_{k}$ and $x$ such that $\displaystyle\lim_{k\to +\infty} I_{k}=\{x\}$ (with respect to the Hausdorff metric). Thus $f^{n_{k}}(I_{k})$ will meets $U_1$ and $U_2$, for $k$ large enough and so it meets the boundary $\partial U_{1}$ in a point $b_k$, for $k$ large enough. As $\partial U_1$ is finite, one can assume that $b_k=b$ for infinitely many $k$. Therefore $f^{-n_{k}}(b)\cap I_{k}\neq \emptyset$, for infinitely many $k$. It follows that $x\in \alpha_{f}(b)$. As $b\notin \textrm{P}(f)$, then $\alpha_{f}(b)$ is a minimal set of $f$ (Theorem \ref{PAMM1}). So $\omega_{f}(x) = \alpha_{f}(b)$ and hence $x\in \omega_{f}(x)$. A contradiction.\\
	
	\textit{Case 2}: $x\in \overline{P(f)}$.
	Let $(x_{n})_{n\geq 0}\subset \textrm{P}(f)$ be a sequence  converging to $x$ and set $p_{n} = \textrm{Per}(x_{n})$ the period of $x_n$, $n\geq 0$. Then $(p_{n})_{n\geq 0}$ in unbounded: otherwise, $x\in \textrm{P}(f) \subset \textrm{R}(f)$, a contradiction. So we can assume that $(p_{n})_{n\geq 0}$ goes to infinity. Since $\alpha_{f}(x)$ is a minimal set, $x\notin \alpha_{f}(x)$ (because otherwise, $x\in \omega_{f}(x)$, a contradiction). Now as in Case 1, let $U_{2}$ be an open neighborhood of $\alpha_{f}(x)$ and $U_{1}$ an open neighborhood of $x$ with finite boundary and disjoint from $U_{2}$. Let $(I_{k})_{k\geq 0}$ be a sequence of arcs joining $x$ and $x_{k}$ such that $\displaystyle\lim_{k\to +\infty} I_{k}=\{x\}$ (with respect to the Hausdorff metric). Thus $f^{-p_{k}}(I_{k})$ meets $\partial U_{1}$ infinitely many times and thus there exists $b\in \partial U_{1}$ such that $f^{p_{k}}(b)\in I_{k}$. This implies that $x\in \omega_{f}(b)$. As $\omega_{f}(b)$ is a minimal, so $x\in \textrm{R}(f)$. A contradiction.
\end{proof}		
\medskip

\begin{cor}\label{c513}	If $\Omega(f)$ is finite or countable, then $\Omega(f)= \mathrm{P}(f)$.
\end{cor}

\begin{proof} Since $\Omega(f)=\textrm{AP}(f)$ (Theorem \ref{t41}), so $\textrm{AP}(f)$ is at most countable and hence every minimal set for $(X, f)$ is a periodic orbit. Therefore $\textrm{AP}(f)= \textrm{P}(f)$ and then Corollary \ref{c513} follows.
\end{proof}
\medskip

In \cite{Coven}, Coven and Nitecki have shown that for continuous maps $f$ of the closed interval $[0,1]$, $\Omega(f)=\{x\in [0,1]: x\in \alpha_{f}(x)\}.$ Later, it is extended to graph maps in \cite[Corollary 1]{MS}. However, for dendrite maps, it does not holds (see \cite{SQLTX}). The following theorem extend the above result to monotone maps on regular curves.
\medskip

\begin{thm}\label{RR}
	Let $X$ be a regular curve and $f$ a self monotone mapping of $X$. Then $\Omega(f)= \{x\in X: x\in \alpha_{f}(x)\}$.
\end{thm}

\begin{proof}
	Let $x\in \Omega(f)$. By Theorem \ref{t41}, $x\in \omega_{f}(x)$ and hence there exists a negative orbit $(x_{n})_{n\geq 0}$ of $x$ such that $(x_{n})_{n\geq 0}\subset \omega_{f}(x)$. By minimality of $\omega_{f}(x)$, we have $x\in \omega_{f}(x)=\alpha_f((x_n)_{n\geq 0})\subset \alpha_{f}(x)$. Conversely, let $x\in X$ such that $x \in \alpha_{f}(x)$. We can assume that $x\notin \textrm{P}(f)$ (if $x\in \textrm{P}(f)$ then $x\in \Omega(f)$ and we are done). So by Theorem \ref{PAMM1}, $\alpha_{f}(x)$ is a minimal set and $x\in \alpha_{f}(x)$. Therefore $\omega_{f}(x)=\alpha_{f}(x)$. Hence $x\in \omega_{f}(x)$ and so $x\in \textrm{R}(f) = \Omega(f)$.
\end{proof}
\medskip

\begin{cor}
	Let $X$ be a regular curve and $f$ a self monotone mapping of $X$. If $x\in \Omega(f)$, then :\\
	\rm{(i)} $\omega_{f}(x)\subset \alpha_{f}(x)$.\\
	\rm{(ii)} If $x\notin \mathrm{P}(f)$, then $\omega_{f}(x) = \alpha_f((x_n)_{n\geq 0}) = \alpha_{f}(x)$, for any negative orbit $(x_{n})_{n\geq 0}$ of $x$.
\end{cor}

\begin{proof}
	\rm{(i)} Since $x\in \Omega(f)$, so $x\in \alpha_{f}(x)$ (Theorem \ref{RR}) and hence $\omega_{f}(x) \subset \alpha_{f}(x)$ (since by Proposition \ref{p43}, $\alpha_{f}(x)$ is closed and $f$-invariant).\\
	\rm{(ii)} If $x\notin \textrm{P}(f)$, then $\alpha_f((x_n)_n) = \alpha_{f}(x)$ (Theorem \ref{PAMM1}). Again by Theorem \ref{PAMM1} and from (i), we get $\alpha_{f}(x) = \omega_{f}(x)$.
\end{proof}
\medskip
	
	 In the following example, we show that Theorem \ref{t41} cannot be extended to rational curves; we construct a self monotone map $f$ on a rational curve $D$ (it is in fact a dendroid) such that the inclusion is strict: $\textrm{R}(f) \subsetneq \Omega(f)$.
	
 Notice also that Theorem \ref{t41} does not hold even for continuous maps on intervals; for instance; we may extend the map $g: L\to L$ (defined in Example \ref{CEM2}) into a continuous map $h$ of $[0,1]$ so that $T_{1}\in \Lambda(h)\subset \Omega(h)$ but $T_{1}\notin \textrm{R}(h)$.
		
	\medskip
	
	\begin{exe}\label{CEM2}
		\rm{ The idea of the construction consists to define a continuous map $g$ on a countable compact set $L$ satisfying $\mathrm{R}(g) \subsetneq \Omega(g)$ and then extends it to a monotone map $f$ defined on a dendroid $D$.\\

			$\bullet$ Let $\sigma$ be the shift map on the Cantor set $\{0,1\}^{\mathbb{N}}$ endowed with the metric $d$ defined as follows: For $x, y\in \{0,1\}^{\mathbb{N}}$, $d(x,y) = 2^{-N(x,y)}$, where $N(x,y) =\min\{n\in \mathbb{N}: x_{n}\neq y_{n}\}$ with $N(x,x) = +\infty, \forall x\in \{0,1\}^{\mathbb{N}}$. We denote by:\\
			
			We let  $T_{0}=\overline{0}$ and $T_{i} = (000000\dots\underbrace{1}_{i^{th}-\textrm{position}}0\dots), \textrm{ for }~ i>0$. Set\\
			
			$\mathcal{Z} = (\underbrace{10}\underbrace{100}\underbrace{1000}\dots\underbrace{10\dots0}_{n-zero}1\dots)$,\;
				
			$O_{-}(\mathcal{Z})=\{T_{-i}=(000\dots \underbrace{\mathcal{Z}}_{i^{th}
					-\textrm{position}}): i\geq 1\}$,
				
				$L = O_{\sigma}(\mathcal{Z}) \cup \{T_{i}: i\geq 0\} \cup O_{-}(\mathcal{Z})$.
				
			\medskip
			
			Observe that $(T_{-i})_{i\geq 1}$ converges to $T_0$. Let us show that $\omega_{\sigma}(Z) = \{T_{i}: i\geq 0\}$.

 For each $n\in \mathbb{N}$, we let $k_{n}=\sum_{k=1}^{n-1}k+n=\frac {n(n-1)}{2}+n$. It is clear that $(k_{n})_{n\in \mathbb{N}}$ is an increasing sequence. Observe that $\sigma^{k_{n}}(\mathcal{Z})$ gives the first apparition of the number one after $n$ zeros. Then for each $n\in \mathbb{N}$,   $d(\sigma^{k_{n}}(\mathcal{Z}),\bar{0})<2^{-n}$ and so $\bar{0}\in \omega_{\sigma}(\mathcal{Z})$. Let $i\geq 2$ and suppose that $n\geq i$. Then $\sigma_{k_{n}+(n-i+1)}(\mathcal{Z})$ gives $i-1$ zero before the first apparition of one. Then $d(\sigma^{k_{n}+(n-i+1)}(\mathcal{Z}),T_{i})<2^{-(n+1)}$, so $T_{i}\in \omega_{\sigma}(Z), \forall i\geq2$.  For $i=1$, we have $T_{1} = (100\dots)=\sigma(T_{2})\in \omega_{\sigma}(\mathcal{Z})$. In result, $\{\bar{0}\}\cup \{T_{i},\;i\in\mathbb{N}\}=\{T_{i}: i\geq 0\}\subset \omega_{\sigma}(\mathcal{Z})$. Conversely, let $y=(y_{n})_{n\in\mathbb{N}}\in \omega_{\sigma}(\mathcal{Z})$ and suppose that there exist $i, j \in \mathbb{N} $ such that  $i<j$ and $y_{i}=y_{j}=1$, let  $(n_{s})_{s\geq 0}$ be an increasing sequence of integer be such that $\lim_{s\to +\infty}\sigma^{n_{s}}(\mathcal{Z})=y$. Then there exists $s_{1}\in \mathbb{N}$ such that $\forall s\geq s_{1}$,  we have $n_{s}>k_{j+1}$,  so  $d(\sigma^{n_{s}}(\mathcal{Z}),y)\geq 2^{-j}, \forall s\geq s_{1}$ ( we have only a single one in the first block with length $j$ in $\sigma^{n_{s}}(\mathcal{Z})$  but, we have two ones in the first block with length $j$ in $y$). So every point of $\omega_{\sigma}(\mathcal{Z})$ contains at most only one. Consequently, $\omega_\sigma(\mathcal{Z})=\{T_{i},\;i\in \mathbb{N}\}\cup \{\bar{0}\}=\{T_{i}: i\geq 0\}$.\\
			
		We conclude that $L$ is a countable compact set. We let $g=\sigma_{\mid L}$. Then $g(T_i)= T_{i-1}$, for any $i\geq 1$. Pick a homeomorphic copy of $L$ in $[0,1]\times \{1\}$. For $i\in \mathbb{Z}$, let $I_{i}$ be an arc joining $T_{i}$ and $T_{0}$ so that $S:=\displaystyle\bigcup_{i\in \mathbb{Z}}I_{i}$ is an infinite star centered at $T_{0}$. in particular $I_0= \{T_0\}$.

For each $k\in \mathbb{Z}$, let $i_{k}\in \mathbb{Z}$ such that $d(g^{k}(\mathcal{Z}), S) = d(g^{k}(\mathcal{Z}), I_{i_{k}})$ and let $J_{k}$ be an arc joining $T_{0}$ and $g^{k}(\mathcal{Z})$ which is included in the closed ball $B_F(I_{i_{k}}, d(g^{k}(\mathcal{Z}),I_{i_{k}}))$ of the plan.
We may assume that for any $k\in \mathbb{Z}$, $S\cap J_{k} = \{T_{0}\}$ and for any $k\neq l,\; J_{k}\cap J_{l} = \{T_{0}\}$. Let $D = S \bigcup \big (\displaystyle\bigcup_{k\geq 0}J_{k}\big)$. Clearly $D$ is a countable union of arcs, so by (\cite[Theorem 6, page 286]{Kur}) $D$ is a rational curve which is a dendroid.
			
			$\bullet $ We extend $g$ into a map $f$ on $D$ so that $f(J_{k}) = J_{k+1}$ and $f_{| J_{k}}$ is affine, for each $k\in \mathbb{Z}$ and $f(I_{i})=I_{i-1}$ for $i\geq 1$. Then $f(I_{1})=\{T_{0}\}$ and $f_{\mid L}=g$.
			It is clear that $f$ is a pointwise monotone map on $D$ and onto, thus monotone. Moreover $\textrm{R}(f) \subsetneq \Omega(f)$, since for instance $T_{1}\in \Lambda(f)\subset \Omega(f)$ but $T_{1}\notin \textrm{R}(f)$.}

\end{exe}
\begin{figure}[h!]
	\centering
	\includegraphics[width=12cm, height=10cm]{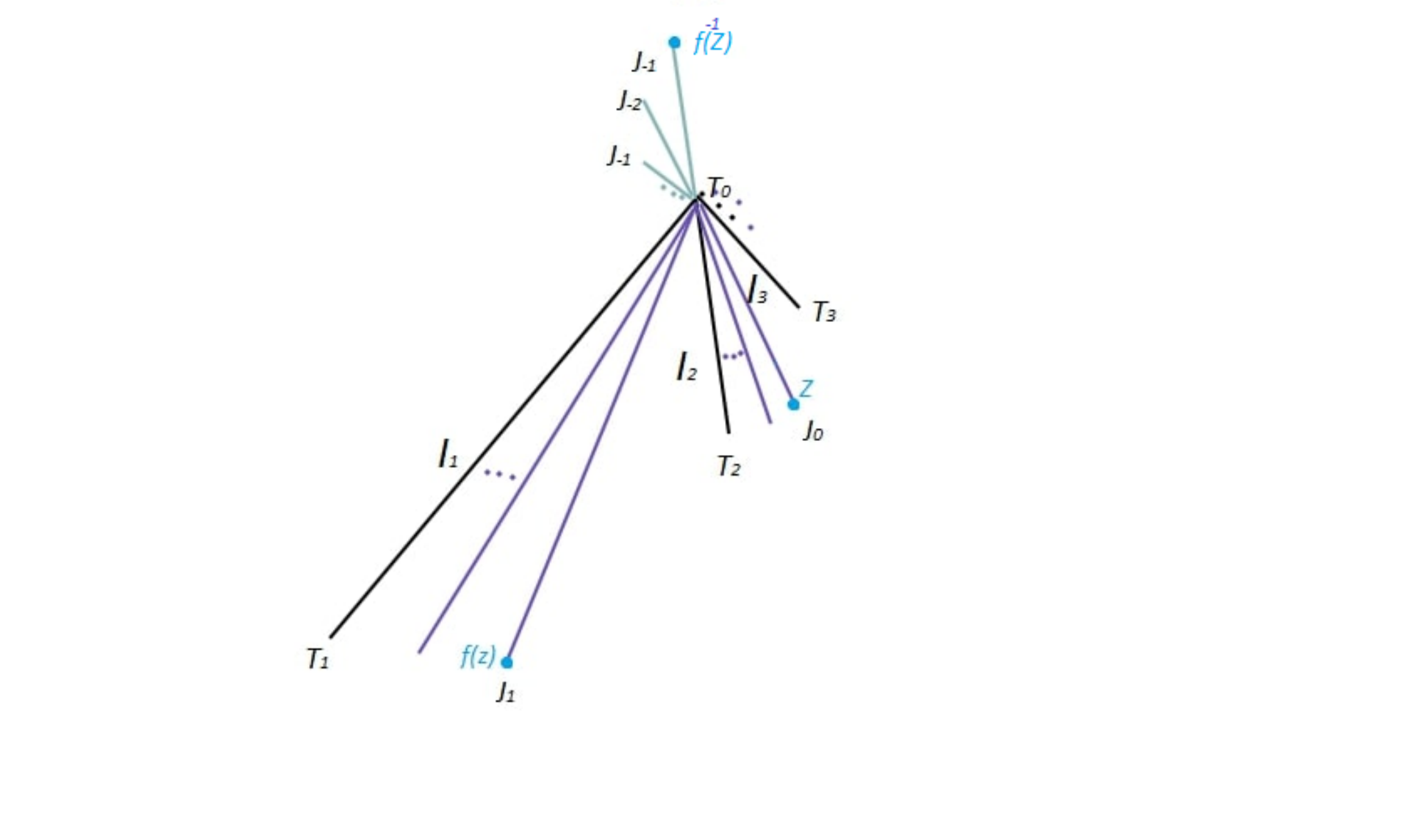}
	\caption{The map $f$ on the dendroid $D$}
\end{figure}

\section{\bf The space of minimal sets with respect to the Hausdorff metric}
 The main result of this section is to prove the following theorem.

\begin{thm}\label{LM}
Let $X$ be a regular curve and $f$ a monotone self mapping of $X$. Then any limit (with respect to the Hausdorff metric) of a sequence of minimal sets $(M_{n})_{n\geq 0}$ is a minimal set.
\end{thm}

 \begin{proof}
 Let $(M_{n})_{n\geq 0}$ be a sequence a sequence of minimal sets converging to some $M\subset X$. Then $f(M)=M$. Since each $M_{n}\subset \textrm{R}(f)$, so by Theorem \ref{t41}, $M\subset \textrm{R}(f)=\textrm{AP}(f)$. Therefore $M=\displaystyle\bigcup_{i\in I}M_{i}$, where $M_{i}$ is a minimal set for any $i\in I$. Assume that $M$ is not a minimal set. Then $\textrm{card}(I)>1$.\\

 \textit{Claim}: $M\subset \textrm{P}(f)$. \ Suppose that $M\nsubseteq \textrm{P}(f)$, so there exists $i_{0}\in I$ such that $M_{i_{0}}$ is infinite. Set $P= M_{i_{0}}$ and $N=M_{j_{0}}$, where $j_{0}\neq i_{0}$. Then $P$ and $N$ are disjoint minimal sets and so does $\mathcal{B}(P)$ and $N$. By Theorem \ref{PAMM1}, $\mathcal{B}(P)$ is a closed set. So let $U$ an open set with finite boundary $\partial U$ of cardinality $k$ such that $\mathcal{B}(P)\subset U$ and $N\cap \overline{U}=\emptyset$.\\
 Recall that $P \cup N\subset M=\displaystyle\lim_{n\to +\infty}M_{n}$, so fix some $y\in P,\; z\in N$ and let $y_{n}, \ z_{n}\in M_{n}$ such that $(y_{n})_{n\geq 0}$ (resp.  $(z_{n})_{n\geq 0}$) converges to $y$ (resp. $z$). Let $(I_{n})_{n\geq 0}$ be a sequence of arcs joining $y$ and $y_{n}$ so that $\{y\}=\displaystyle\lim_{n\to +\infty}I_{n}$.\\
 Since $M_{n}$ is a minimal set, we can find $s_{n}\geq 0$ such that $d(f^{-s_{n}}(y_{n}),z_{n})\leq \frac{1}{n}$, for every $n\geq 1$. Hence for $n$ large enough, $f^{-s_{n}}(y_{n}) \nsubseteq U$ and $f^{-s_{n}}(I_{n}) \nsubseteq U$. On the other hand, $f^{-s_{n}}(I_{n})$ meets $P$ at some point of $f^{-s_{n}}(y)$. As $f$ is monotone, thus for $n$ large enough, $f^{-s_{n}}(I_{n})\cap \partial U\neq \emptyset$. As $\partial U$ is finite, then there exists $z\in \partial U$ such that $z\in f^{-s_{n}}(I_{n})$, for infinitely many $n$. Hence $y\in \omega_{f}(z)$ (since $d(f^{s_n}(z),y)\leq \textrm{diam}(I_n) + d(y_n, y)$). By minimality of $\omega_{f}(z)$ (see Theorem \ref{PAMM1}), we get $\omega_{f}(z)=P$. Hence $z\in \mathcal{B}(P) \cap \partial U$. A contradiction with $\mathcal{B}(P) \subset U$. The proof is complete.\\

 \textit{ Proof of Theorem \ref{LM}}. As proven by the claim above, we have $M\subset \textrm{P}(f)$. Since any minimal set $M_{n}$ has the weakly incompressibility property (Lemma \ref{wc}), so does its limit $M$ by \cite[Proposition 3.1]{FWCP}. Since $M\subset \textrm{P}(f)$, it is uncountable, otherwise, $M$ will have an isolated point $a$, so by Lemma \ref{l27}, $M=O_{f}(a)$ is a periodic orbit. A contradiction. From $M= \displaystyle \cup_{n\in \mathbb{N}}\textrm{Fix}(f^{n})\cap M$, there exist $N>0$ and an infinite sequence $(z_{n})_{n\geq 0}$ in Fix$(f^{N})\cap M$ with disjoint orbits converging to some $z\in \textrm{Fix}(f^{N})\cap M$. As $O_{f}(z_{0})\cap O_{f}(z)=\emptyset$, there exists an open set $V_{z}$ with finite boundary such that $O_{f}(z)\subset V_{z}$ and $O_{f}(z_{0})\cap \overline{V_{z}} = \emptyset$.
Since $(z_{n})_{n\geq 0}$ converges to $z$, there is some $p\in \mathbb{N}$ such that for any $n\geq p,\; O_{f}(z_{n})\subset V_{z}$. For each $n\geq p$, let $(I_{n,q})_{q\geq 0}$ be a sequence of arcs joining $z_{n}$ and $z_{n,q}\in M_{q}$ so that $(I_{n,q})_{q\geq 0}$ converges to $\{z_{n}\}$  as done in the proof of the claim above. As $M_{q}$ is minimal and $M_{q} \nsubseteq U$, we can find $s_{q}>0$ such that $f^{-s_{q}}(z_{n,q}) \nsubseteq U$. Recall that $O_{f}(z_{n})\subset U$, so $f^{-s_{q}}(I_{n,q})\cap \partial U \neq \emptyset$, for any $q\geq 0$. By the same argument as in the proof of the claim above, we can find, for each $n\geq p$, $b_{n}\in \partial U$ such that $\omega_{f}(b_{n})=O_{f}(z_{n})$. This leads to a contradiction since $\{O_{f}(z_{n}) : n\geq p\}$ is an infinite family of disjoint periodic orbits.
 \end{proof}

\section{\bf On the continuity of limit maps $\omega_{f}$ and $\alpha_{f}$}
 In this section, we shall investigate the continuity of the limit maps: \\ $\omega_{f}: X \to 2^{X}; ~x\to \omega_{f}(x)$ and $\alpha_{f}: X_{\infty}\to 2^{X}; ~x\to \alpha_{f}(x)$.

\begin{thm}\label{WACC}
The maps $\omega_{f}$ and $\alpha_{f}$ are continuous everywhere except may be at the periodic points of $f$.
\end{thm}

We use the following lemmas.

\begin{lem} $($\cite[Lemma 4.2]{Ay3}$)$ \label{az}
Let $X$ be a hereditarily locally connected continuum, $F\subsetneq X$ a closed subset and $(O_{n})_{n\geq 0}$ a sequence of open subsets of $X$ such that $F = \displaystyle\bigcap_{n\geq 0}\overline{O_{n}}$. Then $\displaystyle\lim_{n\rightarrow +\infty} \mathrm{Mesh}(O_{n}\setminus F)=0$.
\end{lem}
	
 \begin{lem}\label{COMEGA} The restriction map $(\omega_{f})_{| \Omega(f)}$ is continuous.
 \end{lem}

 \begin{proof}
 Let $(x_{n})_{n\geq 0}$ a sequence of $\Omega(f)$ converging to some $x\in \Omega(f)$. By Theorem \ref{t41}, $x_{n}\in \omega_{f}(x_{n})$ and by Theorem \ref{PAMM1}, $(\omega_{f}(x_{n}))_{n\geq 0}$ is a sequence of minimal set, thus by Theorem \ref{LM}, any limit point of that sequence should be a minimal set. Observe that $x$ belongs to any limit point of the sequence $(\omega_{f}(x_{n}))_{n\geq 0}$, therefore any limit point of the sequence $(\omega_{f}(x_{n}))_{n\geq 0}$ is a minimal set that contains $x$, therefore it has to be $\omega_{f}(x)$.\\
 \end{proof}
 \medskip

 Let $(X, f)$ be a dynamical system. For $x\in X$, we denote by \\ $S_{x} = \underset{i\geq 0}\cup f^{-i}(x)$.

 \begin{lem}\label{Orbit periodique d} Let $x, y\in X$ such that $O_{f}(x)$ and $O_{f}(y)$ are two disjoint orbits. Then $S_{x}$ and $S_{y}$ are disjoint.
 \end{lem}

 \begin{proof}
 	Assume that $S_{x}\cap S_{y}\neq \emptyset$. Then we can find two positive integers $m\leq n$ and some $z\in f^{-m}(x)\cap f^{-n}(y)$. Therefore $f^{n}(z)\in O_{f}(x)\cap O_{f}(y)$, which will leads to a contradiction.
 \end{proof}
 \medskip

\begin{lem}\label{WAC}
Let $X$ be a regular curve and $f$ a monotone self mapping of $X$. The limit maps $\omega_{f}$ and $\alpha_{f}$ are continuous at any point of $X\setminus \Omega(f)$.
\end{lem}

\begin{proof}
First note that $X \backslash \textrm{R}(f) = X \backslash \Omega(f)$ is a open set disjoint from $\overline{P(f)}$. Thus for any point $a\in X \backslash \Omega(f)$, $\alpha_{f}(a)$ and $\omega_{f}(a)$ are minimal sets. Let $x\in X \backslash \textrm{R}(f)$:\\
$\bullet$ Continuity of $\omega_{f}$:
 Assume that $\omega_{f}$ is not continuous at $x$ and set $M=\omega_{f}(x)$, then we can find a sequence $(x_{n})_{n\geq 0}$ that converges to $x$ such that $\omega_{f}(x_{n})$ converges to $A\neq M$. By Theorem \ref{LM}, $A$ is a minimal set, thus $A\cap M=\emptyset$. Let $m\geq 0$ and $V_{A,m}\subset B(A,\frac{1}{m})$ be an open set with finite boundary such that $M\cap \overline{V_{A,m}}=\emptyset$. Let $(I_{n})_{n\geq 0}$ be a sequence of arcs joining $x$ and $x_{n}$, where $(I_{n})_{n\geq 0}$ converge to $\{x\}$ with respect to the Hausdorff metric. Recall that $A = \displaystyle \lim_{n\to +\infty}\omega_{f}(x_{n})$, so for $n$ large enough we can find some $p_{n}$ such that $f^{p_{n}}(x_{n})\in V_{A,m}$ and $f^{p_{n}}(x)\notin V_{A,m}$. Then $f^{p_{n}}(I_{n})$ will meets $\partial V_{A,m}$ and so $x\in \alpha_{f}(b_{m})$ for some $b_{m}\in \partial V_{A,m}$. Since $x\notin \Omega(f)$, we conclude that $\alpha_{f}(b_{m})$ is not a minimal set, thus By Theorem \ref{PAMM1}, $b_{m}\in \textrm{P}(f)$.\\

 \textit{Claim}: The sequence $(b_{m})_{m\geq 0}$ satisfies the follow properties:\\
 (i) For any $m\geq 0, \;b_{m}\in \textrm{P}(f)$ and $x\in \alpha_{f}(b_{m})$.\\
 (ii) The sequence $(O_{f}(b_{m}))_{m\geq 0}$ converges to $A$.\\
 (iii) By passing to a subsequence if needed, the family $(O_{f}(b_{m}))_{m\geq 0}$ is pairwise disjoint.\\

\textit{Proof of the claim}:
 (i) is already proven above, for the prove of (ii) observe that any limit point of $(b_{m})_{m\geq 0}$ is a point of $A$ and since $\{b_{m}:\; m\geq 0\}\cup A\subset \Omega(f)$, so by Lemma \ref{COMEGA} ends the proof of (ii).\\
 (iii) Clearly for any $m\geq 0$, $O_{f}(b_{m})\cap A=\emptyset$. Let $b_{m_{0}}=b_{0}$ and $O_{1}$ be an open neighborhood of $A$ such that $O_{f}(b_{m_{0}})\cap O_{1}=\emptyset$. As $(b_{m})_{m\geq 0}$ converges to $A$ we can find $m_{1}>0$ such that $b_{m_{1}}\in O_{1}$. Thus $O_{f}(b_{0})$ and $O_{f}(b_{1})$ are disjoint.\\
 For $N\geq 0,\; (O_{f}(b_{m_{i}}))_{0\leq i\leq N}$ is defined and pairwise disjoint. Let $O_{N+1}$ an open neighborhood of $A$ such that\\ $\displaystyle\bigcup_{0\leq i\leq N}O_{f}(b_{m_{i}})\cap O_{N+1}=\emptyset $. We can find $m_{N+1}>m_{N}$ such that $b_{m_{N+1}} \in O_{N+1}$ and then $(O_{f}(b_{n_{i}}))_{0\leq i\leq N+1}$ is pairwise disjoint. We have defined by indication the subsequence $(b_{m_{i}})_{i\geq 0}$ so that $(O_{f}(b_{m_{i}}))_{i\geq 0}$ are pairwise disjoint. This ends the proof of the claim.
\smallskip

Now $x\notin \textrm{R}(f)$, so $x\notin A$. Let $V$ be an open neighborhood of $A$ such that $x\notin \overline{V}$ and set $k=\textrm{card}(\partial V)$. By (iii) there exist $N\geq 0$ such that for any $m\geq N,\; \omega_{f}(b_{m})=O_{f}(b_{m})\subset V$. By (i), we may find for each $m\geq 0$ some $s_{m}\geq 0$ such that $f^{-s_{m}}(b_{m}) \nsubseteq V$. Fix $N\leq m_{0}<m_{2}<\dots m_{k}$. By Lemma \ref{Orbit periodique d} and monotonicity of $f$, the family $(f^{-s_{m_{i}}}(b_{m_{i}}))_{0\leq i\leq k}$ is a pairwise disjoint family of connected subset of $X$. Observe that for each $0\leq i\leq N,\; f^{-s_{m_{i}}}(b_{m_{i}})\nsubseteq V$ and $f^{-s_{m_{i}}}(b_{m_{i}})$ meets $V$ at some point of $O_{f}(b_{m_{i}})$. Therefore for each $0\leq i\leq N,\; f^{-s_{m_{i}}}(b_{m_{i}})$ meets $\partial V$. This leads to a contradiction.\\

$\bullet$ Continuity of $\alpha_{f}$.

Set $M=\alpha_{f}(x)$. By Theorem \ref{PAMM1}, $M$ is a minimal set, assume that $(\alpha_{f}(x_{n}))_{n\geq 0}$ converges to $L\neq M$, then $L \nsubseteq M$, so let $t\in L\setminus M$.\\
Let $V_{t}$ be an open neighborhood of $t$ with finite boundary such that $M\cap \overline{V_{t}} = \emptyset$ and $(I_{n})_{n\geq 0}$ a sequence of arcs joining $x$ and $x_{n}$ converging to $\{x\}$ with respect to the Hausdorff metric. Since $t\in L = \displaystyle \lim_{n\to +\infty}\alpha_{f}(x_{n})$, for $n$ large enough, we can find some $p_{n}\geq 0$ such that $f^{-p_{n}}(x_{n})\subset V_{t}$ and $f^{-p_{n}}(x)\nsubseteq V_{t}$. Then by monotonicity of $f$ $f^{-p_{n}}(I_{n})$ will meets $\partial V_{t}$ and thus $x\in \omega_{f}(b)$ for some $b\in \partial V_{t}$. By Theorem \ref{t41}, $\omega_{f}(b)\subset \textrm{R}(f)$, hence $x\in \textrm{R}(f)$. A contradiction.
\end{proof}

\textit{Proof of Theorem \ref{WACC}.} Let $x\in X\setminus \textrm{P}(f)$. If $x\notin \Omega(f)$, then by Lemma \ref{WAC}, $\omega_{f}$ is continuous at $x$. So assume that $x\in \Omega(f)\backslash \textrm{P}(f)$ and set $M = \omega_{f}(x)$. Then $x\in M$ and $M$ is an infinite minimal set.\\

$\bullet$  \textit{Continuity of $\alpha_{f}$ on $X\setminus \textrm{P}(f)$}:
By Theorem \ref{PAMM1}, $\alpha_{f}(x)=M$. Assume that there is some sequence $(x_{n})_{n\geq 0}$ of $X$ converging to $x$ such that $\alpha_{f}(x_{n})$ converges to $L$ and $L\neq M$. By minimality of $M$, $L\nsubseteq M$. Assume that there exists $y\in L\setminus M$. We distinguish two cases:\\

\textit{Case 1}: \textit{For infinitely many $n\geq 0$, $x_{n}\in \textrm{P}(f)$.} Since $x$ in non periodic point, so per($(x_{n}))_n$ is unbounded. By choosing $(x_{n}))_{n\geq 0}$ with pairwise distinct periods  per($(x_{n}))_n$, one can assume that $(O_{f}(x_{n}))_{n\geq 0}$ is pairwise disjoint. Therefore by Lemma \ref{Orbit periodique d}, the family $(S_{x_n})_{n\geq 0}$ is pairwise disjoint. Let now $V_{y}, V_{M}$ be two disjoint open neighborhoods of $y, M$ with finite boundary. By Lemma \ref{COMEGA}, the restriction map $(\omega_{f})_{| \Omega(f)}$ is continuous. Therefore as $x\in \textrm{R}(f)\backslash \textrm{P}(f)$, we conclude that the sequence $O_{f}(x_{n}))_{n\geq 0}$ converges to $\omega_{f}(x)=M$. So one can assume that for any $n\geq 0$, we have $O_{f}(x_{n})\subset V_{M}$. 
As $y\in \displaystyle\lim_{n\to +\infty}\alpha_{f}(x_{n})$, then for each $n\geq 0$, one can find $m_{n}>0$ such that $f^{-m_{n}}(x_{n})\cap V_{y}\neq \emptyset$. So let $S_{m_{n},n}$ be the connected component of $S_{n}$ containing $f^{-m_{n}}(x_{n})$. Hence $(S_{m_{n},n})_{n\geq 0}$ is a family of pairwise disjoint connected sets each of which meets $V_{M}$ and $V_{y}$ (since $x_{n}$ is a periodic point such that $O_{f}(x_{n})\subset V_{M}$). This leads to a contradiction since $\partial(V_{M})$ is finite.\\

\textit{Case 2:} \textit{For $n$ large enough, $x_{n}\notin \textrm{P}(f)$}. In this case, by Theorem \ref{PAMM1}, $\alpha_{f}(x_{n})$ is a minimal set. Thus by Theorem \ref{LM}, $L$ is also a minimal set and then $M\cap L=\emptyset$. So let $V_{M}$ and $V_{L}$ be two disjoint neighborhoods of $M, L$ with finite boundary and set $k= \textrm{card}(\partial(V_{M}))$. As $(x_{n})_{n\geq 0}$ converges to $x$, then we can find $k+1$ pairwise disjoint arcs $I_{0},I_{1},\dots, I_{k}$, each $I_{j}$ joins $f^{j}(x)$ and $f^{j}(x_{N})$, for some $N$ large enough and satisfying $\alpha_{f}(x_{N}) \subset V_{L}$. Then we can find $s\geq 0$ such that for any $i\geq s$, $f^{-i}(x_{N})\subset V_{L}$. 
Consider the family $(f^{-(s+k)}(I_{j}))_{0\leq j\leq k}$, which is a family of pairwise disjoint connected sets, each of which meets $V_{L}$ at $f^{-(s+k)}(x_{N})$ and meets $V_{M}$ at some point of $f^{-(s+k)}(f^{j}(x))$, this will lead to a nonempty intersection with $\partial(V_{M})$ for each $0\leq j\leq k$. A contradiction. \\
This ends the proof of the continuity of $\alpha_{f}$.
\medskip

$ \bullet$  \textit{Continuity of $\omega_{f}$}. Assume that there is some sequence $(x_{n})_{n\geq 0}$ of $X$ converging to $x$ such that $\omega_{f}(x_{n})$ converges to $L$ and $L \neq M$. As $\omega_{f}(x_{n})$ is a minimal set, for any $n\geq 0$, thus by Theorem \ref{LM}, $L$ is also a minimal set and then $M\cap L=\emptyset$.\\

\textit{Claim 1.} For any neighborhood $V_{L}$ of $L$, there exists $z\in V_{L}\cap \textrm{P}(f)$ such that $M\subset \alpha_{f}(z)$.\\

\textit{Proof}. Fix $V_{L}$ some neighborhood of $L$. Let $U_{M}$ and $U_{L}$ be two disjoint neighborhoods of $M, L$ with finite boundary such that  $U_{L}\subset V_{L}$. Set $k=\textrm{card} (\partial(U_{L}))$. As $(x_{n})_{n\geq 0}$ converges to $x$ and $O_{f}(x)$ is infinite, then for each $0\leq j\leq k$, we can find a sequence of arcs $(I_{n,j})_{n\geq 0}$ joining $f^{j}(x_{n})$ and $f^{j}(x)$ converging to $\{f^{j}(x)\}$ such that for any $n,m\geq 0$ and for any $0\leq i< j\leq k$, $I_{n,i}\cap I_{m,j}=\emptyset$. Let $\eta=\min\{d(f^{p}(x),f^{q}(x)),\; 0\leq p<q\leq k\}$. Clearly $\eta>0$ and we may assume that $\varepsilon=\inf_{n\geq 0}\{d(I_{n,p},I_{n,q}): 0\leq p<q\leq k\}>0$.\\
We can assume that for any $n\geq 0,\; \omega_{f}(x_{n})\subset U_{L}$. Thus for each $n\geq 0$, we can find $m_{n}>0$ such that $\{f^{m_{n}}(f^{j}(x_{n})): 0\leq j\leq k\}\subset U_{L}$. Therefore for each $n\geq 0,\; 0\leq j\leq k,\; f^{m_{n}}(I_{n,j})\cap \partial(U_{L}) \neq \emptyset$. Then for each $n\geq 0$, there exists $z_{n}\in \partial (U_{L})$ and $0\leq j_{1,n}<j_{2,n}\leq k$ such that $z_{n}\in f^{m_{n}}(I_{n,j_{1,n}}) \cap f^{m_{n}}(I_{n,j_{2,n}})$. Hence $f^{-m_{n}}(z_{n})\cap I_{n,j_{p,n}}\neq \emptyset$, for $p\in \{0,1\}$. Recall that $z_{n}\in \partial(O_{L})$ which is finite and $0\leq j_{n,1}<j_{n,2}\leq k$. Hence there exists $z\in \partial(U_{L})$ and $0\leq j_{1}<j_{2}\leq k$ such that $f^{-m_{n}}(z)\cap I_{n,j_{p}}\neq \emptyset$, for infinitely many $n$, $p\in \{0,1\}$. Therefore for infinitely many $n$, we have that $\textrm{diam}(f^{-m_{n}}(z))\geq \varepsilon $. Hence $z\in \textrm{P}(f)$. Moreover as $f^{-m_{n}}(z)\cap I_{n,j_{p}}\neq \emptyset$, then $O_{f}(x)\cap \alpha_{f}(z)\neq \emptyset$ and hence $M\subset \alpha_{f}(z)$. This ends the proof of Claim 1.
\medskip

Now by Claim 1, we may find a sequence $(z_{n})_{n\geq 0}$ of periodic points with disjoint orbits converging to some point $l\in L$ such that $M\subset \alpha_{f}(z_{n})$. By Lemma \ref{COMEGA}, $(O_{f}(z_{n}))_{n\geq 1}$ converges to $L$. Moreover for each $n\geq 1$, we can find $s_{n}>0 $ such that $d(f^{-s_{n}}(z_{n}),M)\leq \frac{1}{n}$. As $O_{f}(z_{n})\subset \textrm{P}(f)$, then $O_{f}(z_{n}) \cap f^{-s_{n}}(z_{n}) \neq \emptyset $. Therefore we can find $N\geq 0$ such that for any $n\geq N$, we have $\textrm{diam}(f^{-s_{n}}(z_{n}))\geq \frac{d(M,L)}{2}$.  Recall that $(z_{n})_{n\geq 0}$ is a sequence of periodic points with disjoint orbits, hence $(f^{-s_{n}}(z_{n}))_{n\geq N}$ is a non null family of pairwise disjoint subcontinua. A contradiction with the fact that $X$ is a regular curve and hence finitely Suslinean. This ends the proof of the continuity of $\omega_{f}$. \qed

\section {\bf On special $\alpha$-limit sets}
In \cite{sep alpha1}, Hero introduced another kind of limit sets, called the\textit{ special} $\alpha$-limit sets. He considered the union of the $\alpha$-limit sets over all backward orbits of $x$.

\begin{defn} Let $X$ be a metric compact space, $f: X\to X$ a continuous map and $x\in X$. The special $\alpha$-limit set of $x$, denoted $s\alpha_{f}(x)$, is the union $s\alpha_{f}(x) = \displaystyle\bigcup \alpha_f((x_n)_{n\geq 0})$ taken over all negative orbits $(x_n)_{n\geq 0}$ of $x$.
\end{defn}

Notice that we have the following equivalence:

(i) $\alpha_{f}(x)\neq \emptyset$, (ii) $s\alpha_{f}(x)\neq \emptyset$, (iii) $x\in X_{\infty}$.

In particular, if $f$ is onto, then $s\alpha_{f}(x)\neq \emptyset$ for every $x\in X$.
\medskip

It is clear that $s\alpha_{f}(x)\subset \alpha_{f}(x)$. The inclusion can be strict. Hero \cite{sep alpha1} provided an example of a continuous map on the interval for which the inclusion $s\alpha_{f}(x) \subset \alpha_{f}(x)$ is strict. Even, one can provide an onto monotone interval map. Indeed, consider the monotone map $g: [0,1]\longrightarrow [0,1]:  x\longmapsto \max\{0,2x-1\}$. We see that
$s\alpha_{g}(0) = \{0, 1\} \subsetneq \alpha_{g}(0) = [0,1]$.

In (\cite[Theorem 3.3 and Corollary 3.11]{kms}), Kolyada et al. provided an example of a map on a subset of $\mathbb{R}^{2}$ where a special $\alpha$-limit set is not closed. Recently, Hant\'{a}kov\'{a} and Roth proved in (\cite{hr}, Theorem 37) that a special $\alpha$-limit set for interval map is always Borel, and in fact both $F_{\sigma}$ and $G_{\delta}$. Furthermore, they provided a counterexample of an interval map with a special $\alpha$-limit set which is not closed, this disproves the conjecture 1 in \cite{kms}. However they showed that a special $\alpha$-limit set is closed for a piecewise monotone interval map. Jackson et al. (\cite{jmr}) proved that a special $\alpha$-limit set is always analytic (i.e. a continuous image of a Polish space) and provide an example of a map of the unit square with special $\alpha$-limit set not a Borel set. Here we show that, for a monotone map on a regular curve, the special $\alpha$-limit set is always closed.
   %
\\
	
\subsection{\bf Relation between Nonwandering sets, $\alpha$-limit and Special $\alpha$-limit sets} The aim of this paragraph is to prove the following theorem.

\begin{thm}\label{S=OSS}
	Let $X$ be a regular curve and $f$ a monotone self mapping of $X$. Then for every $x\in X_{\infty}$, we have that \ $s\alpha_{f}(x) = \alpha_{f}(x)\cap \Omega(f)$.
\end{thm}
\medskip

First, we derive from Theorem \ref{PAMM1} the following corollary.

\begin{cor}\label{c62} Let $X$ be a regular curve and $f$ a monotone self mapping of $X$. Then or any $x\in X_{\infty}$: \\
	(i) $ s\alpha_{f}(x) \subset \Omega(f)$, \\
	(ii) $s\alpha_{f}(x)$ is a union of minimal sets.  \\
	(iii) if $x\in X_{\infty}\setminus \mathrm{P}(f)$, then $s\alpha_{f}(x) = \alpha_{f}(x)$ is a minimal set.
\end{cor}

\begin{proof} Let $x\in X_{\infty}$. By Theorem \ref{PAMM1}, (2), $\alpha_{f}((x_{n})_{n\geq 0})$ is a minimal set for every negative orbit $(x_{n})_{n\geq 0}$ of $x$, thus (i) and (ii) follow. Assume now that $x\in X_{\infty}\setminus \textrm{P}(f)$, again by Theorem \ref{PAMM1}, (2), we have $\alpha_{f}((x_{n})_{n\geq 0})=\alpha_{f}(x)$ is a minimal set for any negative orbit $(x_{n})_{n\geq 0}$ of $x$, thus (iii) follows.
\end{proof}

The proof of Theorem \ref{S=OSS} needs the following lemma.

\begin{lem}\label{TMSF}
Let $X$ be a regular curve and $f$ a monotone self mapping of $X$. If $x\in \mathrm{P}(f)$, then every minimal set $M\subset \alpha_{f}(x)$ is a periodic orbit.
\end{lem}

\begin{proof}
Let $x\in \textrm{P}(f)$ and $M$ an infinite minimal set such that $M\subset \alpha_{f}(x)$. Fix some $y\in M\subset \alpha_{f}(x)$, we can find an increasing sequence of integers $(m_{n})_{n\geq 0}$ and $x_{n}\in f^{-m_{n}}(x)$ such that $(x_{n})_{n\geq 0}$ converges to $y$. Clearly for any $n\geq 0$, we have $\omega_{f}(x_{n})=O_{f}(x)$. By Theorem \ref{WACC}, $\omega_{f}$ is continuous at $y$, therefore we get $\omega_{f}(y)= M = O_{f}(x)$, which is finite. A contradiction.
\end{proof}

Recall that given a subset $A$ of a topological space $X$, the arc connected component $C$ of $a\in A$ is defined as
$$C=\{y\in A: \textrm{there exists an arc } I\subset A \textrm{ joining } a \textrm{ and } y\}.$$

\begin{lem}\label{ARC et LARC}
\rm{Let $X$ be a regular curve and $A$ a subset of $X$. Then the following hold:\\
\rm{(i)} Every arc connected component of $A$ is closed in $A$.\\
\rm{(ii)} If $C$ is arcwise connected, then it is locally arcwise connected.}
\end{lem}

\begin{proof} Since a regular curve is finitely suslinean continuum, so the proof of (i) results from \cite[Corollary 2.2]{ARC} and the proof of (ii) results from \cite[ Corollary 5.5]{ARC2}.
\end{proof}

\begin{lem}\label{l66} Let $X$ be a regular curve and $f$ a monotone self mapping of $X$. If $x\in \mathrm{P}(f)$ and $O_{f}(x)\subsetneq \alpha_{f}(x)$, then $\Omega(f)\neq X$.
\end{lem}
	
\begin{proof}
Assume that $\Omega(f)=X$, by Theorem \ref{t41}, $\Omega(f) = \textrm{R}(f)$, thus $\textrm{R}(f)=X$. Observe that for some $n\geq 0, f^{-n}(x) \nsubseteq O_{f}(x)$, if not $O_{f}(x)=\alpha_{f}(x)$. So let $t\in X\setminus O_{f}(x)$ and $n>0$ such that $f^{n}(t)=x$, then $\omega_{f}(t) = O_{f}(x)$ and $t\notin O_{f}(x)$. This will lead to a contradiction since $t\in \textrm{R}(f)$.
\end{proof}
		
\begin{proof}[Proof of Theorem \ref{S=OSS}]
If $x\in  X_{\infty}\setminus \textrm{P}(f)$, then by Corollary \ref{c62}, $s\alpha_{f}(x) = \alpha_{f}(x)\subset \Omega(f)$ and so $ s\alpha_{f}(x) = \alpha_{f}(x)\cap \Omega(f)$. Now assume that $x\in \textrm{P}(f)$. The inclusion $s\alpha_{f}(x)\subset \alpha_{f}(x)\cap \Omega(f)$ follows from  Corollary \ref{c62}. Now let $y\in \alpha_{f}(x)\cap \Omega(f)$. By Lemma \ref{TMSF}, $\omega_{f}(y)$ is a periodic orbit and so $y\in \textrm{P}(f)$. If $O_{f}(y)\cap O_{f}(x)\neq \emptyset$, then $y\in O_{f}(x)\subset s\alpha_{f}(x)$. Now assume that $O_{f}(y)\cap O_{f}(x)=\emptyset$, then $\eta=d(O_{f}(x), O_{f}(y))>0$ and $O_{f}(x) \subsetneq\alpha_{f}(x)$. Therefore by Lemma \ref{l66}, $\Omega(f)\subsetneq X$. By Lemma \ref{az}, we can find an open set $U$ such that $\Omega(f)\subset U$ and \textrm{Mesh}$(U\setminus \Omega(f))<\frac{\eta}{2}$. Denote by $A=\displaystyle\bigcup_{n\geq 0}f^{-n}(x)\cup O_{f}(x)\cup O_{f}(y)$,  observe that $f(A)\subset A$.\\

\textit{Claim}: $A$ has a finite number of arc connected components.\\
Denote by $q$ the period of $x$. We have $$A = \displaystyle\bigcup\limits_{i=0}^{q-1}\big(\displaystyle\bigcup_{n\geq 0}f^{-i}(f^{-qn}(x))\big)\cup O_{f}(x)\cup O_{f}(y).$$ Since $f$ is monotone then for each $n\geq 0,\ f^{-qn}(x)$ is a subcontinuum of the regular curve $X$, thus arcwise connected, moreover it contains $x$. We conclude that $\displaystyle\bigcup_{n\geq 0}(f^{-qn}(x))$ is an arcwise connected subset of $X$. Again since $f$ is monotone and $\displaystyle\bigcup_{n\geq 0}(f^{-qn}(x))$ is arcwise connected then for each $0\leq i\leq q-1,\; \displaystyle\bigcup_{n\geq 0}f^{-i}(f^{-qn}(x))$ is also arcwise connected. Indeed for $a,b\in \displaystyle\bigcup_{n\geq 0}f^{-i}(f^{-qn}(x))$, we can find an arc $I$ joining $f^{i}(a)$ and $f^{i}(b)$ in $\displaystyle\bigcup_{n\geq 0}(f^{-qn}(x))$, thus we can find some arc $J$ in $f^{-i}(I)\subset \displaystyle\bigcup_{n\geq 0}f^{-i}(f^{-qn}(x))$ joining $a$ and $b$. We conclude that $A$ can be written as a finite union of arcwise connected subset of $X$, this end the proof of the claim.
\medskip
  Recall that $y\in \alpha_{f}(x)$ then we can find an increasing sequence of integer $(m_{n})_{n\geq 0}$ and $x_{n}\in f^{-m_{n}}(x)$ so that $(x_{n})_{n\geq 0}$ converges to $y$. For each $n\geq 0,\; x_{n}\in A$ which has a finite number of arc connected component by the claim above, namely $C_{1},\dots C_{m}$, so we may find $1\leq j\leq m$ such that for each $n\geq 0,\; x_{n}\in C_{j}$ so that $y\in \overline{C_{j}}$. By Lemma \ref{ARC et LARC}, (i), $C_{j}\cup\{y\}$ is arcwise connected. Therefore by Lemma \ref{ARC et LARC}, (ii), locally arcwise connected. Let then $I_{n}\subset \big(C_{j}\cup\{y\}\big) \cap B(y,\frac{1}{n})$ be an arc in $C_{j}\cup\{y\}$ joining $y$ and $x_{n}\in f^{-m_{n}}(x)$ so that $(I_{n})_{n\geq 0}$ converges to $\{y\}$.
We have for each $n\geq 0 \; \{x,f^{m_{n}}(y)\}\subset f^{m_{n}}(I_{n})\subset A\subset \big(X\setminus \Omega(f)\big) \cup \{O_{f}(x),O_{f}(y)\}$. Therefore for each $n\geq 0,\; f^{m_{n}}(I_{n})$ is an arcwise connected set of $X$ meeting $\Omega(f)$ only at $O_{f}(x),\;O_{f}(y)$. Therefore $f^{m_{n}}(I_{n}) \nsubseteq U$.
Thus for each $n\geq 0$, there exists $z_{n} \in f^{m_{n}}(I_{n}) \cap \partial U \subset X\setminus \Omega(f)$. As $\partial U$ is finite there exists $z\in \partial U$ such that $z_{n} = z$, for infinitely many $n$. Hence $f^{-m_{n}}(z) \cap I_{n} \neq \emptyset$ for infinitely many $n\geq 0$. Since $z\in \partial U \subset X\setminus \textrm{P}(f)$, so by Theorem \ref{PAMM1}, $\alpha_{f}(z)$ is a minimal set containing $y$. Thus $z\in X_{\infty}$ and $\alpha_{f}(z) = O_{f}(y)$. Recall that for each $n\geq 0$,  $z_{n}\in A$ and hence $z\in A$. Then for some $p>0$, we have $f^{p}(z)=x$, therefore any negative orbit of $z$ can be completed into a negative orbit of $x$ hence $s\alpha_{f}(z) \subset s\alpha_{f}(x)$. By Theorem \ref{PAMM1}, we have $s\alpha_{f}(z) = \alpha_{f}(z) = O_{f}(y)$ and so $y\in s\alpha_{f}(x)$.
\end{proof}

\begin{cor}\label{c53}
Let $X$ be a regular curve and $f$ a monotone self mapping of $X$. Then for every $x\in X$, $s\alpha_{f}(x)$ is a closed set.
\end{cor}

\begin{proof}
The proof follows from Theorem \ref{S=OSS} \;(since $\Omega(f)$ and $\alpha_{f}(x)$ are closed).
\end{proof}

The following proposition improves Theorem \ref{RR}.

\begin{prop}\label{RRS}  Let $X$ be a regular curve and $f$ a monotone self mapping of $X$. Then
	$\Omega(f) = \{x\in X: x\in s\alpha_{f}(x)\}$.
\end{prop}

\begin{proof}
	Let $x\in X$ such that $x\in s\alpha_{f}(x)$. By Corollary \ref{c62}, $s\alpha_{f}(x)$ is a union of minimal sets, thus $x\in \Omega(f)$. Conversely, let $x\in \Omega(f)$. Then by Theorem \ref{t41}, $x\in \omega_{f}(x)$. Thus $x$ has a negative orbit $(x_{n})_{n\geq 0}\subset \omega_{f}(x)$. Therefore by minimality of $\omega_{f}(x)$ (cf. Theorem \ref{PAMM1}), we have $\alpha_{f}((x_{n})_{n\geq 0}) = \omega_{f}(x)$. Hence $x\in \alpha_{f}((x_{n})_{n\geq 0})\subset s\alpha_{f}(x)$. The proof is complete.
\end{proof}
\medskip

\subsection{\bf Further results on special $\alpha$-limit sets }

\begin{thm}[Continuity of special $\alpha$-limit maps]\label{tcs} Let $X$ be a regular curve and $f$ a monotone self mapping of $X$. Then the special $\alpha$-limit map $s\alpha_{f}$ is continuous everywhere except may be at the periodic points of $f$.
\end{thm}

\begin{proof}
	Let $x\in X_{\infty}\setminus \textrm{P}(f)$ and $(x_{n})_{n\geq 0}$ be a sequence of $X_{\infty}$ converging to $x$. By Theorem \ref{WACC}, the sequence $\alpha_{f}(x_{n})_{n\geq 0}$ converges to $\alpha_{f}(x)$. Let $F$ be any limit point of $s\alpha_{f}(x_{n})_{n\geq 0}$ with respect to the Hausdorff metric. Clearly $F\subset \alpha_{f}(x)$, moreover by Corollary \ref{c53}, $F$ is an invariant closed subset of $X$. By Theorem \ref{PAMM1}, (2), $\alpha_{f}(x)$ is a minimal set, therefore $F=\alpha_{f}(x)$. It turn out that $\alpha_{f}(x)$ is the unique limit point of the sequence $s\alpha_{f}(x_{n})_{n\geq 0}$. We conclude that the sequence $s\alpha_{f}(x_{n})_{n\geq 0}$ converges to $\alpha_{f}(x)$ (since $2^{X}$ is compact). Again by Theorem \ref{PAMM1}, (2), $\alpha_{f}(x)=s\alpha_{f}(x)$ and the result follows.
\end{proof}

 Note that it may happened that the maps $\omega_{f},\alpha_{f}$ and $s\alpha_{f}(x)$ are discontinuous at some periodic point: For example consider the map $f: X\to X$ given in \cite[Example 5.10]{MONOTONE}. 
 
 \begin{exe}$($\cite[Example 5.10]{MONOTONE}$)$ \label{e616} There exists a monotone map $f$ on an infinite star $X$ centered at a point $z_{0}\in \mathbb{R}^{2}$ and beam $I_{n},\; n\geq 0$ with endpoints  $z_{0}$ and $z_{n}$ satisfying the following properties.
 	\medskip
 	
 	\begin{itemize}
 		\item[(i)] For any $x\in X\setminus \{z_{n}: n\geq 0\}$ we have $\omega_{f}(x)=\{z_{0}\}$ and $\alpha_{f}(x) =\{z_{n_{x}}\}$, where $n_{x}\geq 1$ such that $x\in I_{n_{x}}$.
 		\item[(ii)] For any $n\geq 1$ we have $\alpha_{f}(z_{n})=\omega_{f}(z_{n})=\{z_{n}\}$.
 		\item[(iii)] $\omega_{f}(z_{0})=\{z_{0}\},\; \alpha_{f}(z_{0})=X$.
 	\end{itemize} 
 \end{exe}
 
 We have
  $\omega_{f}$ is discontinuous at $z_{n}$, for every $n\geq 0$ and $\alpha_{f}$ (resp. $s\alpha_{f}$) is discontinuous at $z_{0}$.\\

Denote by $\textrm{SA}(f)$ (respectively $\textrm{A}(f)$) the union of all special $\alpha$-limit sets (respectively, all $\alpha$-limit sets) of a map $f$. In \cite[Corollary 2.2]{kho}, it was shown that $\textrm{R}(f)\subset \textrm{SA}(f)$ holds for general dynamical system $(X,f)$.  For continuous map $f$ on graphs, it was shown that $\textrm{SA}(f)\subset\overline{\textrm{R}(f)}\subset \Lambda(f)$ (see \cite{sxl}, cf. \cite{bdlo}, \cite{sep alpha1}) and that there are continuous maps $f$ on dendrites with $\textrm{SA}(f)\nsubseteq\overline{\textrm{R}(f)}$ (cf. \cite{stsxq}). In the following theorem, we extend the above results to regular curves when we restricted to monotone maps.

\begin{thm}
Let $X$ be a regular curve and $f$ a monotone self mapping of $X$, the following hold:\\
    \rm{(i)} $\mathrm{SA}(f)=\mathrm{R}(f)$.\\
    \rm{(ii)}  $\mathrm{SA}(f)$ and $\mathrm{A}(f)$ are closed. \\
     \rm{(iii)} If $\mathrm{P}(f) = \emptyset$, then $\mathrm{SA}(f) = \mathrm{R}(f)=\mathrm{A}(f)$.\\
\end{thm}

\begin{proof} Recall that by Theorem \ref{t41}, $\Omega(f) = \textrm{R}(f)$.
Assertion (i) follows immediately from Corollary \ref{c62} and Proposition \ref{RRS}.

\medskip
(iii): Assume now that $\textrm{P}(f)=\emptyset$. Then by Corollary \ref{c62}, $s\alpha_{f}(x)=\alpha_{f}(x)$ is a minimal set, for any $x\in X_{\infty}$, so assertion (iii) follows.

\medskip
(ii): By (i), $\Omega(f)=\textrm{R}(f)=\textrm{SA}(f)$ and so $\textrm{SA}(f)$ is closed. Let us show that $\mathrm{A}(f)$ is closed.
Let $(x_{n})_{n\geq 0}$ be a sequence of $\textrm{A}(f)$ converging to some $x\in X$. We can assume that $x\notin \Omega(f)$ since otherwise $x\in \Omega(f) \subset \textrm{A}(f)$ and we are done. For each $n\geq 0$, let $b_{n}\in X_{\infty}$ such that $x_{n}\in \alpha_{f}(b_{n})$. Suppose that $x\notin \textrm{A}(f)$, then by passing to a subsequence if needed, one can assume that the sequence $(b_{n})_{n\geq 0}$ is infinite (since otherwise $x\in \textrm{A}(f)$).

Since $x\notin \Omega(f)$, so for $n$ large enough, $x_{n}\notin \Omega(f)$ and therefore $\alpha_{f}(b_{n})$ is not a minimal set. Therefore by Theorem \ref{PAMM1}, $b_{n}\in \textrm{P}(f)$. So by passing to a subsequence if needed, the family $(O_{f}(b_{n}))_{n\geq 0}$ is pairwise disjoint. Let $U$ be an open neighborhood of $x$ with finite boundary such that $\overline{U}\cap \Omega(f)=\emptyset$. As $x_{n}\in \alpha_{f}(b_{n})$ and the sequence $(x_{n})_{n\geq 0}$ converges to $x$, there exists $N\geq 0$ such that for any $n\geq N$, we can find $s_{n}\geq 0$ such that $f^{-s_{n}}(b_{n})\cap U \neq \emptyset$. Since $b_{n}\in \textrm{P}(f)$, we have $f^{-s_{n}}(b_{n}) \nsubseteq U$. Since $f$ is monotone, it follows that $f^{-s_{n}}(b_{n})\cap \partial U \neq \emptyset$, for any $n\geq N$. By Lemma \ref{Orbit periodique d} the family $(f^{-s_{n}}(b_{n}))_{n\geq N}$ is pairwise disjoint and as proven above $f^{-s_{n}}(b_{n})\cap \partial U \neq \emptyset$, for each $n\geq N$, this will lead to a contradiction since $\partial U$ is finite.
\end{proof}
\medskip

\begin{rem}\rm{
  If  $\mathrm{P}(f)\neq \emptyset$, we have the inclusion $\textrm{SA}(f)\subset \textrm{A}(f)$ which can be strict for monotone maps on regular curves. Indeed, consider the map $f_{\infty}: S_{\infty}\to S_{\infty}$ of Example \ref{r59}. We have that  \\ $\textrm{SA}(f_{\infty})= \textrm{R}(f_{\infty})= \textrm{P}(f_{\infty}) \subsetneq \textrm{A}(f_{\infty}) = S_{\infty}$.}
\end{rem}

\begin{thm}\label{WA}
Let $X$ be a regular curve and $f$ a monotone self mapping of $X$. Then for any $x\in \mathrm{P}(f),\; s\alpha_{f}(x)$ is either a finite union of periodic orbits or an infinite sequence of periodic orbits converging to $O_{f}(x)$.
\end{thm}

\begin{proof}
Let $x\in \textrm{P}(f)$. By Corollary \ref{c62} and Lemma \ref{TMSF}, $s\alpha_{f}(x)$ is a union of periodic orbits.

 Assume that $s\alpha_{f}(x)$ is infinite and that there is an accumulation point $a$ of $s\alpha_{f}(x)$ such that $a\notin O_{f}(x)$. We can find an infinite sequence $(a_{n})_{n\geq 0}$ of $s\alpha_{f}(x)$ converging to $a$. 
 By Theorem \ref{LM}, $(O_{f}(a_{n}))_{n\geq 0}$ converges to a minimal set $L$. As $a_{n}\in \textrm{P}(f)$, then $a\in L$ and hence $L = O_{f}(a)$. Then for $i$ large enough, say $i\geq p$, $$d\left(O_{f}(a_{i}), ~O_{f}(x)\right)\geq \frac{d\Big(O_{f}(x), ~O_{f}(a)\Big)}{2}.$$ Since $\Omega(f) \subsetneq X$ (Lemma \ref{l66}), so by Lemma \ref{az}, there is an open set $O$ such that $\textrm{card}(\partial O)=k$ and $$\textrm{Mesh}(O\setminus \Omega(f))<\frac{d\left(O_{f}(x), ~O_{f}(a)\right)}{2}.$$ As done in the proof of Theorem \ref{S=OSS}, the set $$A = \displaystyle\cup_{n\geq 0}f^{-n}(x)\cup \{O_{f}(a_{i}), ~p\leq i\leq k+p\}$$ is $f$-invariant and has a finite number of arc connected components. For each $p\leq i\leq k+p$, let $C_{i}$ be the arc connected component of $A$ containing $a_{i}$.
Since $\;a_{i}\in \alpha_{f}(x)$, for each $p\leq i\leq k+p$, we can find an increasing sequence of integers $(m_{n,i})_{n\geq 0}$ and a sequence $(x_{n,i})_{n\geq 0}$ which converges to $a_{i}$, where $x_{n,i}\in f^{-m_{n,i}}(x)$. By the same arguments as in the proof of Theorem \ref{S=OSS}, we can assume that $x_{n,i}\in C_{i}$ for infinitely many $n\geq 0$. Now as the $a_{i}$ are pairwise distinct and the $C_{i}$ are locally arcwise connected, $p\leq i\leq k+p$, so for $n$ large enough, we may find a sequence of pairwise disjoint arcs $(I_{n,i})_{n\geq 0}$ in $C_{i}$ joining $x_{n,i}$ and $a_{i}$ which converges to $\{a_{i}\}$. Choose $N$ large enough and set $I_{i}=I_{N,i}$, $p\leq i\leq k+p$. So $(I_{i})_{p\leq i\leq k+p}$ is a family of $k+1$ pairwise disjoint arcs in $A$ joining $a_{i}$ and $x_{N,i}\in f^{-m_{N,i}}(x)$.
Set
$$\displaystyle \eta:= \min\{d(I_{i}, I_{j}): p\leq i<j\leq k+p\}>0$$ and $$m>\max \{m_{N,i}: p\leq i\leq k+p\}.$$ Since $A$ is $f$-invariant, $f^{m}(I_{i}) \subset A$. Moreover $f^{m}(I_{i})$ is a subcontinuum of $X$ which meets the orbits of $x$ and $a_{i}$. As $$A\cap \Omega(f) = O_{f}(x) \cup \{O_{f}(a_{i}): p\leq i\leq k+p\},$$ then we can find an arc $J\subset f^{m}(I_{i})$ joining $x^{\prime}\in O_{f}(x)$ and $b_{i}\in O_{f}(a_{i})$ such that $J\setminus \{x^{\prime}, b_{i}\} \subset A\setminus \Omega(f)$. Obviously $J\setminus \{x^{\prime}, b_{i}\}$ is connected and $\textrm{diam}(J\setminus \{x^{\prime}, b_{i}\})\geq \frac{d\left(O_{f}(x), ~O_{f}(a)\right)}{2}$. Therefore $J \nsubseteq O$ and $J \cap O \neq \emptyset$. Thus $J\cap \partial O \neq \emptyset$ and therefore $f^{m}(I_{i})\cap \partial O \neq \emptyset$, for each $p\leq i\leq k+p$. It follows that there exists $z_{m}\in \partial O$ such that $z_{m}\in f^{m}(I_{i})\cap f^{m}(I_{j})$, for some $p\leq i<j\leq k+p$. Hence $\textrm{diam}(f^{-m}(z_{m}))>\eta$. As $\partial O$ is finite, there exists $z\in \partial O$ such that $z=z_{m}$ for infinitely many $m$, thus $\displaystyle\limsup_{s\to +\infty}\textrm{diam}(f^{-s}(z))>0$. This implies that the family $(f^{-n}(z))_{n\geq 0}$ is not pairwise disjoint (otherwise it will be a null family), and hence for some $q_1<q_2$, we have $f^{-q_1}(z)\cap f^{-q_2}(z) \neq \emptyset$. Thus $z=f^{q_2-q_1}(z)$ and so $z\in \textrm{P}(f)$. As $z\in \partial O$ and $\textrm{P}(f)\subset \Omega(f) \subset O$, then $z\in X\setminus \textrm{P}(f)$. A contradiction.\\

 We conclude that any accumulation point of $s\alpha_{f}(x)$ is a point of $O_{f}(x)$. It turns out that $s\alpha_{f}(x)$ is a compact space having a finite set of accumulation points, hence it is countable.
\end{proof}
\medskip
	
	The following corollary is in contrast to the properties of $\omega$-limit sets.
	
\begin{cor}\label{c12}
Let $X$ be a regular curve and $f$ a monotone self mapping of $X$ and $x\in X_{\infty}$. If $s\alpha_{f}(x)$ contains an isolated point, then $s\alpha_{f}(x)\subset \mathrm{P}(f)$.
\end{cor}

\begin{proof}
If $x\notin \textrm{P}(f)$, then $s\alpha_{f}(x)$ is a minimal set and hence it is a periodic orbit since it contains $z$ which is an isolated point in it. If $x\in \textrm{P}(f)$, then from Lemma \ref{TMSF}, we have $s\alpha_{f}(x)\subset \textrm{P}(f)$.
\end{proof}

 \begin{cor}\label{c512}
 A countable $s\alpha$-limit set for a monotone map on a regular curve is a union of periodic orbits.
 \end{cor}

 \begin{proof}
 Let $x\in X_{\infty}$ such that $s\alpha_{f}(x)$ is countable. We may that $s\alpha_{f}(x)$ is not minimal; otherwise $s\alpha_{f}(x)$ will be a finite minimal set and we are done. Therefore by Theorem \ref{PAMM1}, (2), $x\in \textrm{P}(f)$ and by Theorem \ref{WA}, the result follows.
 \end{proof}

\begin{rem} \rm{ Notice that Corollaries \ref{c12} and \ref{c512} extend those of \cite[Theorem 20 and Corollary 21]{hr} to monotone maps on regular curves.}
\end{rem}	
\medskip

\begin{rem} \label{r58} \ \rm{\begin{itemize}
			\item[(1)] A special $\alpha$-limit set can be totally periodic and infinite for monotone map on a regular curve; for example, in Example \ref{e616}, we have $s\alpha_{f}(z_{0})$ is infinite and composed of fixed points. This is in contrast to some properties of $\omega$-limit sets (cf. \cite[Theorem 2.4]{MONOTONE}).\\
		\item[(2)] In \cite{an}, it was constructed a continuous self-mapping (not monotone) of a dendrite having totally periodic $\omega$-limit set with unbounded periods. In Example \ref{r59} below, we construct analogously, a monotone map on a dendrite having a totally periodic special $\alpha$-limit set with unbounded periods.
			\end{itemize}}
\end{rem}

\begin{exe} \ \label{r59}\rm{
The idea is to slightly change the map $f$ in Example \ref{e616} so that $s\alpha_{f}(z_{0})$ is infinite and composed of periodic orbits with unbounded period. Let $z_{0}$ some point of the plan $\mathbb{R}^{2}$, $N\geq 1$ and denote by $S_{N}$, the $N$-star centered at $z_{0}$. For $N=1$, we define $f_{1}: S_{1} \longrightarrow S_{1}$ identified as the map $f_{1} = g: [0,1]\longrightarrow [0,1]:  x\longmapsto \max\{0,2x-1\}$, where $z_{0}$ and $S_{1}$ play the role of $0$ and $[0,1]$, respectively. Let now $N\geq 1, S_{N}= \bigcup\limits_{k=0}^{N-1} I_{k,N}$, each  $I_{k,N}$ is an arc of $\mathbb{R}^{2}$, where $z_{0}$ is one of its endpoints, we denote by $z_{k,N}$ the other endpoint of $I_{k,N}$ distinct from $z_{0}$. Let $h_{N}: S_{N}\longrightarrow S_{N}$ be the homeomorphism of $S_{N}$ defined as $h_{N}(I_{k,N})= I_{k+1,N}$ and $h_{N}(I_{N-1,N})=I_{0,N}$ in an affine way. We let $f_{N} = h_{N}\circ f_{1}$. Clearly $f_{N}: S_{N}\longrightarrow S_{N}$ is monotone and continuous. Observe that $\textrm{End}(S_{N}) = \{z_{k,N}:\; 0\leq k\leq N-1\}$ is a periodic orbit of period $N$, $s\alpha_{f_{N}}(z_{0}) = \{z_{0},z_{k,N}:\; 0\leq k\leq N-1\}$ and $\alpha_{f_{N}}(z_{0})=S_{N}$.
We may pick $(S_{N})_{N\geq 1}$ in such a way that $S_{\infty}=\displaystyle\bigcup_{n\geq 1}S_{N}$ is an infinite star centered at $z_{0}$. Define the map $f_{\infty}: S_{\infty}\to S_{\infty}$ given by its restriction $(f_{\infty})_{| S_{N}} =f_{N}$. Clearly $f_{\infty}$ is monotone and continuous. Moreover $\alpha_{f_{\infty}}(z_{0})=S_{\infty}$ and $s\alpha_{f_{\infty}}(z_{0})=\{z_{0},z_{k,N}:\; N\geq 1,\; 0\leq k\leq N-1\}$ is totally periodic which contains periodic orbits of period $N$, for any $N\geq 1$.}
\end{exe}
\newpage
\begin{figure}[!h]
	\centering
	\includegraphics[width=10cm, height=6cm]{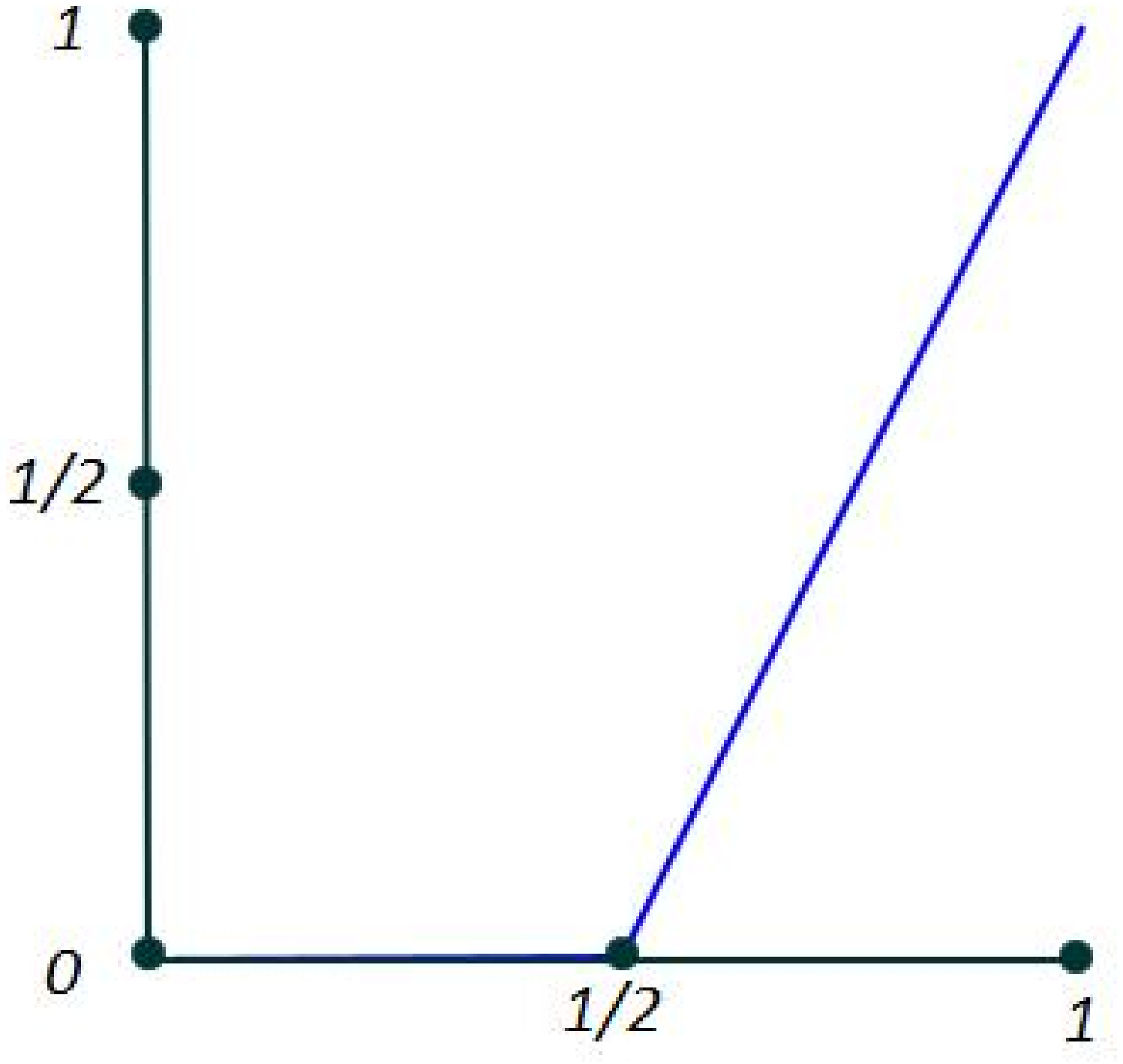}
	\caption{The map $g$}
\end{figure}
\begin{figure}[!h]
	\centering
	\includegraphics[width=20cm, height=11cm]{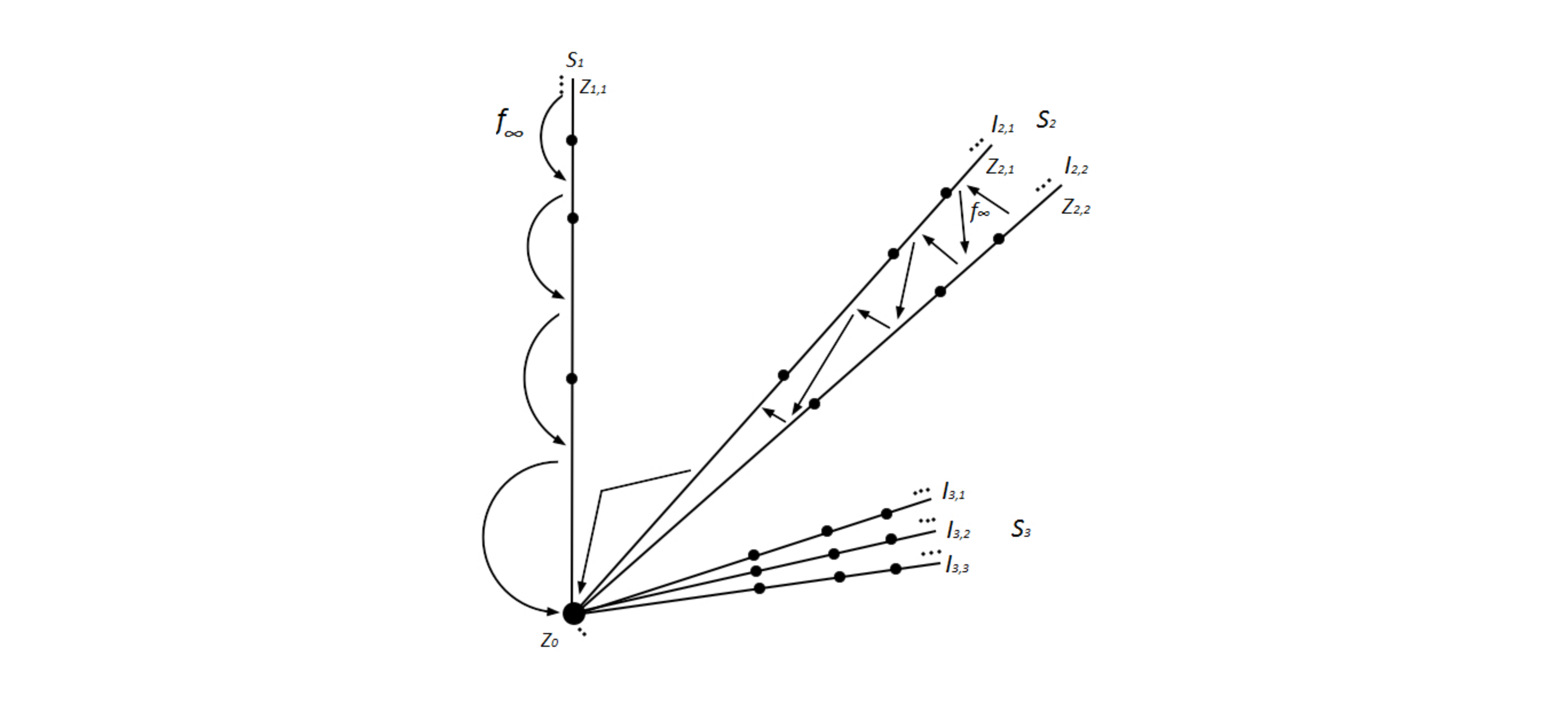}
	\caption{The map $f_{\infty}$}
\end{figure}
\newpage
	
\textit{Acknowledgements.}
This work was supported by the research unit: ``Dynamical systems and their applications'', (UR17ES21), Ministry of
Higher Education and Scientific Research, Faculty of Science of Bizerte, Bizerte, Tunisia.

\bibliographystyle{amsplain}

\medskip

\begin{thebibliography}{9}
\bibitem{a} H. Abdelli, \emph{ $\omega$-limit sets for  monotone local dendrite maps}, Chaos Solitons Fractals, \textbf{71} (2015), 66--72.
\bibitem{HM} H. Abdelli, H. Marzougui, \emph{Invariant sets for monotone local dendrites},
Internat. J. Bifur. Chaos Appl. Sci. Engrg., \textbf{26} (2016), 1650150.
\bibitem{aam} H. Abdelli, H. Abouda and H. Marzougui, \emph{Nonwandering points of monotone local dendrite maps revisited}, Topology Appl.,\textbf{ 250} (2018), 61--73.
\bibitem{FWCP} E.D. Anielloa, T.H. Steele, \textit{The persistence of $\omega$-limit sets defined on compact
	spaces,} J. Math. Anal. Appl. \textbf{413} (2014), 789--799.
\bibitem{an} G. Askri and I. Naghmouchi, \emph{On totally periodic $\omega$-limit sets in regular continua}, Chaos Solitons Fractals \textbf{75} (2015), 91--95.
\bibitem{bgl} F. Balibrea, J.L. Guirao and M. Lampart, \emph{A note on the definition of $\alpha$-limit set}, Appl. Math. Inf. Sci. \textbf{7} (2013), 1929--1932.
\bibitem{bl} F. Balibrea and C. La Paz, \textit{ A characterization of the $\omega$-limit sets of
	interval maps}. Acta Math. Hungar., \textbf{88}(4):291--300, 2000.
\bibitem{bdlo} F. Balibrea, G. Dvorníkov\'{a}, M. Lampart, P. Oprocha, \emph{On negative limit sets for one-dimensional dynamics}, Nonlinear Anal. \textbf{75} (2012), 3262--3267.
\bibitem{Bldevan} P. Blanchard, R.L. Devaney, D. Look, M. Moreno Rocha, P. Seal, S. Siegmund, and D. Uminsky,
	\emph{Sierpinski carpets and gaskets as Julia sets of rational maps}, In Dynamics on the Riemann
	Sphere, European Mathematical Society, Z\"{u}rich, 2006, pp. 97--119.
\bibitem{BLOCK}
	L.S. Block and W. A. Coppel, Dynamics in one dimension, Lecture Notes in Mathematics, 1513, Springer-Verlag, Berlin, 1992.
\bibitem{Coven} E. Coven and Z. Nitecki, \textit{Non-wandering sets of the powers of maps of the interval}, Ergodic Theory Dyn. Syst. \textbf{1} (1981), 9--31.
\bibitem{Ay} A. Daghar (2021): \textit{On regular curve homeomorphisms without periodic points}, J. Difference Equ. Appl., DOI: 10.1080/10236198.2021.1912030.
\bibitem{Ay3} A. Daghar, \textit{Homeomorphisms of hereditarily locally connected continua}. 2021. ⟨hal-03223435v2⟩
\bibitem{MONOTONE} A. Daghar and H. Marzougui. \textit{Limit sets of monotone maps on regular curves}. 2021. arXiv: 2106.12418v1
\bibitem{MONOTONE2} A. Daghar, I. Naghmouchi and M. Riahi, \textit{Periodic Points of Regular Curve Homeomorphisms}, Qual. Theory Dyn. Syst. \textbf{20} (2) (2021).
\bibitem{hr} J. Hant\'{a}kov\'{a} and S. Roth, On backward attractors of interval maps. 2020. arXiv: 2007.10883v2
\bibitem{ARC2} J. Grispolakis, E.D. Tymchatyn, \textit{$\sigma$- connectedness in hereditarily locally connected spaces}, Trans. Amer. Math. Soc., \textbf{253} (1979), 303--315.
\bibitem{sep alpha1} M.W. Hero, \textit{Special $\alpha$-limit points for maps of the interval}, Proc. Amer. Math. Soc. \textbf{116} (1992), 1015--1022.
\bibitem{jmr} S. Jackson, B. Mance and S. Roth, A non Borel special $\alpha$-limit set in the square. 2020. arXiv: 2011.05509v1.
\bibitem{ka} H. Kato, \emph{Topological entropy of monotone maps and confluent maps on regular curves}, Topol. Proc. \textbf{28} (2004), 587--593.
\bibitem{ka2} H. Kato, \emph{Topological entropy of piecewise embedding maps on regular curues}, Ergodic Theory Dyn. Syst. \textbf{26} (2006), 1115--1125.
\bibitem{kms} S. Kolyada, M. Misiurewicz and L. Snoha, \emph{Special $\alpha$-limit sets}, In Dynamics: topology and numbers, 157--173, Contemp. Math., \textbf{744}, Amer. Math. Soc., Providence, RI, (2020).
\bibitem{Kur} K. Kuratowski, Topology, vol.2, Academic Press, New-York, 1968.
\bibitem{kho} M. Fory\'{s}-Krawiec, M.J. Hant\'{a}kov\'{a}, P. Oprocha, On the structure of $\alpha $-limit sets of backward trajectories for graph maps. 2021. arXiv:2106.05539v1.
\bibitem{LFS} A. Lelek, \textit{On the topology of curves. II}, Fund. Math. \textbf{70} (1971), 131--138.
\bibitem{MS} J. Mai, T. Sun, \textit{The $\omega$-limit set of a graph map}, Topology Appl. \textbf{154} (2007) 2306--2311.
\bibitem{Nadler} S.B. Nadler, Continuum Theory: An Introduction, (Monographs and Textbooks in Pure and Applied Mathematics, 158). Marcel Dekker, Inc., New York, 1992.
\bibitem{n} I. Naghmouchi, \textit{Homeomorphisms of regular curves}, J. Difference Equ. Appl., \textbf{23} (2017), 1485--1490.
\bibitem{n2} I. Naghmouchi, \textit{Dynamics of Homeomorphisms of regular curves}, Colloquium Math., \textbf{162} (2020), 263--277.
\bibitem{Nag3} I. Naghmouchi. \textit{Dynamics of monotone graph, dendrite and dendroid maps}. Internat. J. Bifur. Chaos Appl. Sci. Engrg., \textbf{21} (2011), 3205--3215.
\bibitem{Seidler} G.T. Seidler, \emph{The topological entropy of homeomorphisms on one-dimensional continua}, Proc. Amer. Math. Soc., \textbf{108} (1990), 1025--1030.
\bibitem{SQLTX} T. Sun, Q. He, J. Liu, C. Tao, H. Xi, \textit{Non-wandering sets for dendrite maps}, Qual. Theory Dyn. Syst. \textbf{14} (2015), 101--108.
\bibitem{sxl} T. Sun, H. Xi, H. Liang, \emph{Special $\alpha$-limit points and unilateral $\gamma$ limit points for graph maps}, Sci China Math., \textbf{54} (2011), 2013--2018.
\bibitem{stsxq} T. Sun, Y. Tang, G. Su, H. Xi, B. Qin, \emph{Special $\alpha$-limit points and $\gamma$-limit points of a dendrite map}, Qual. Theory Dyn. Syst. \textbf{17} (2018), 245--257.
\bibitem{ARC} E.D. Tymchatyn,  \textit{Characterizations of continua in which connected subsets are arcwise connected}, Trans. Amer. Math. Soc., \textbf{222 }(1976), 377--388.
\end{thebibliography}

\end{document}